\documentclass[11pt]{article}
\usepackage[margin = 1in]{geometry}

\usepackage[affil-it]{authblk}
\usepackage{lipsum}
\usepackage{amsfonts}
\usepackage{graphicx}
\graphicspath{ {images/} }
\usepackage{epstopdf}
\usepackage{dsfont}
\usepackage{mathtools}
\usepackage{empheq}
\usepackage{amsfonts}
\usepackage{amsgen,amsthm,amsmath,amstext,amsbsy,amsopn,amssymb,subfigure,stmaryrd,mathabx,mathrsfs}
\usepackage{comment}
\usepackage[ruled, vlined, lined, commentsnumbered]{algorithm2e}
\usepackage[font=footnotesize,labelfont=bf]{caption}

\usepackage{blkarray}

\usepackage{url}
\usepackage{longtable}
\usepackage{mathtools}
\usepackage{multirow}
\usepackage{array}

\usepackage{enumitem}
\usepackage{hyperref}
\usepackage{natbib}
\usepackage{booktabs}
\usepackage{xcolor}
\usepackage[utf8]{inputenc} 
\usepackage[T1]{fontenc}

\ifpdf
  \DeclareGraphicsExtensions{.eps,.pdf,.png,.jpg}
\else
  \DeclareGraphicsExtensions{.eps}
\fi

\newcommand{\opnorm}[1]{\left\|#1\right\|_{\rm op}}

\renewcommand{\exp}{{\rm{exp}}}

\newcommand{\argmin}{\mathop{\rm arg\min}}

\newcommand{\indi}{{\mathds{1}}}

\newcommand{\wh}{\widehat}
\newcommand{\wt}{\widetilde}
\newcommand{\wc}{\widecheck}

\newcommand{\indc}[1]{\indi\left\{{#1}\right\}}
\newcommand{\norm}[1]{\left\|{#1} \right\|}
\newcommand{\abs}[1]{\left|{#1}\right|}
\newcommand{\Prob}{\mathbb{P}}
\newcommand{\Expect}{\mathbb{E}}
\def\Var{\textsf{Var}}
\def\Cov{\textsf{Cov}}
\newcommand{\bpar}[1]{\left({#1}\right)}

\usepackage{prettyref,soul}

\newtheorem{Theorem}{Theorem}
\newtheorem{Lemma}{Lemma}
\newtheorem{Remark}{Remark}

\newtheorem{Proposition}{Proposition}
\newtheorem{Assumption}{Assumption}

\newrefformat{eq}{equation~\eqref{#1}}
\newrefformat{chap}{Chapter~\ref{#1}}
\newrefformat{sec}{Section~\ref{#1}}
\newrefformat{alg}{Algorithm~\ref{#1}}
\newrefformat{fig} {Figure~\ref{#1}}
\newrefformat{tab}{Table~\ref{#1}}
\newrefformat{rmk}{Remark~\ref{#1}}
\newrefformat{clm}{Claim~\ref{#1}}
\newrefformat{def}{Definition~\ref{#1}}
\newrefformat{cor}{Corollary~\ref{#1}}
\newrefformat{lm}{Lemma~\ref{#1}}
\newrefformat{lem}{Lemma~\ref{#1}}
\newrefformat{lemma}{Lemma~\ref{#1}}
\newrefformat{prop}{Proposition~\ref{#1}}
\newrefformat{app}{Appendix~\ref{#1}}
\newrefformat{ex}{Example~\ref{#1}}
\newrefformat{exer}{Exercise~\ref{#1}}
\newrefformat{soln}{Solution~\ref{#1}}
\newrefformat{cond}{Condition~\ref{#1}}
\newrefformat{as}{Assumption~\ref{#1}}
\newrefformat{th}{Theorem~\ref{#1}}
\newrefformat{ineq}{inequality~\eqref{#1}}

\DeclareMathAlphabet\mathbfcal{OMS}{cmsy}{b}{n}

\definecolor{greengray}{RGB}{0,128,0}

\hypersetup{
    colorlinks=true,
    linkcolor=blue,
    filecolor=blue,      
    urlcolor=cyan,
    citecolor=blue
}

\makeatletter
\newcommand*{\rom}[1]{\expandafter\@slowromancap\romannumeral #1@}
\makeatother

\allowdisplaybreaks

\begin{document}
\title{Confidence Intervals for Linear Models with Arbitrary Noise Contamination}
	
\author[1]{Dong Xie}
\author[1]{Chao Gao\thanks{The research of CG is supported in part by NSF Grants ECCS-2216912 and DMS-2310769, and an Alfred Sloan fellowship.}}
\author[2]{John Lafferty}
\affil[1]{
University of Chicago
}
\affil[2]{
Yale University
}
	\date{}
	\maketitle


\begin{abstract}
We study confidence interval construction for linear regression under Huber's contamination model, where an unknown fraction of noise variables is arbitrarily corrupted. While robust point estimation in this setting is well understood, statistical inference remains challenging, especially because the contamination proportion is not identifiable from the data. We develop a new algorithm that constructs confidence intervals for individual regression coefficients without any prior knowledge of the contamination level. Our method is based on a Z-estimation framework using a smooth estimating function. The method directly quantifies the uncertainty of the estimating equation after a preprocessing step that decorrelates covariates associated with the nuisance parameters. We show that the resulting confidence interval has valid coverage uniformly over all contamination distributions and attains an optimal length of order $O(1/\sqrt{n(1-\epsilon)^2})$, matching the rate achievable when the contamination proportion $\epsilon$ is known. This result stands in sharp contrast to the adaptation cost of robust interval estimation observed in the simpler Gaussian location model.
\end{abstract}


\begin{sloppypar}
\section{Introduction} \label{sec: intro}

This paper considers a linear model
\begin{equation}
y_i=X_i^Tb+z_i,\quad\quad\text{ for } i=1,\cdots,n, \label{eq:mul-lin-comp}
\end{equation}
where the noise variables $\{z_i\}_{i=1}^n$ have an $\epsilon$ fraction of contamination. That is,
\begin{equation}
z_i \overset{iid}\sim (1-\epsilon)N(0,1)+\epsilon Q, \label{eq:noise-sig}
\end{equation}
and they are independent from the $p$-dimensional feature vectors $\{X_i\}_{i=1}^n$. Our goal is to construct a confidence interval for a given coordinate of the regression vector $b\in\mathbb{R}^p$ without imposing any assumption on the contamination distribution $Q$.

The noise distribution (\ref{eq:noise-sig}) is a particular instance of Huber's contamination model \citep{huber1964robust}. Since there is no assumption imposed on $Q$, the contaminated noise variables can just be regarded as arbitrary numbers, which implies the necessity of robust statistical inference. In the regression setting, the linear model (\ref{eq:mul-lin-comp}) with noise following (\ref{eq:noise-sig}) and independent of the design matrix is known as robust regression with oblivious contamination, and has been previously studied by \cite{tsakonas2014convergence,bhatia2015robust,bhatia2017consistent,suggala2019adaptive,gao2020model,d2021consistent,d2021consistent2}. In particular, it was shown by \cite{d2021consistent,d2021consistent2} that an M-estimator with Huber loss achieves the error bound
\begin{equation}
\|\wh{b}-b\|=O_{\mathbb{P}}\left(\sqrt{\frac{p}{n(1-\epsilon)^2}}\right). \label{eq:minimax-rate-l2}
\end{equation}
Moreover, the rate (\ref{eq:minimax-rate-l2}) is optimal in the sense that a matching minimax lower bound can be established for some $Q$ \citep{d2021consistent2}. A notable feature of the rate (\ref{eq:minimax-rate-l2}) is that consistent estimation is still possible even when the contamination proportion $\epsilon$ is close to $1$. This is because in the setting of oblivious contamination, only the response variables $\{y_i\}_{i=1}^n$ are contaminated. In comparison, if the contamination affects both $\{y_i\}_{i=1}^n$ and $\{X_i\}_{i=1}^n$, the optimal estimation rate (\ref{eq:minimax-rate-l2}) becomes $\sqrt{\frac{p}{n}}+\epsilon$ \citep{gao2020robust}, which implies inconsistency unless $\epsilon$ is a vanishing function of the sample size $n$.

Although global estimation of $b$ is well understood, the problem of optimal statistical inference for a given coordinate has remained open. When $\wh{b}$ is computed by M-estimation, its asymptotic normality has been studied by \cite{huber1973robust,yohai1979asymptotic,portnoy1984asymptotic,portnoy1985asymptotic,mammen1989asymptotics} under appropriate conditions on sample size $n$ and dimension $p$. Recent developments using random matrix theory even obtain asymptotic distributions when $n$ and $p$ are proportional \citep{bean2013optimal,el2013robust,karoui2013asymptotic,donoho2016high,lei2018asymptotics}. However, given the lack of assumptions on $Q$, deriving an asymptotic distribution under (\ref{eq:noise-sig}) would be impossible, even when $p=1$.

Beyond M-estimation, another approach to constructing confidence intervals is to use high-probability error bounds for estimation. Given (\ref{eq:minimax-rate-l2}), the optimal error rate for estimating a given coordinate of $b$ is expected to be $\frac{1}{\sqrt{n(1-\epsilon)^2}}$. Suppose that one can construct a robust estimator such that
\begin{equation}
|\wh{b}_1-b_1| = O_{\mathbb{P}}\left(\frac{1}{\sqrt{n(1-\epsilon)^2}}\right).
\end{equation}
Then, the interval $\left[\wh{b}_1-\frac{C}{\sqrt{n(1-\epsilon)^2}},\wh{b}_1+\frac{C}{\sqrt{n(1-\epsilon)^2}}\right]$ is guaranteed to cover $b_1$ with some appropriate constant $C>0$. However, this requires knowledge of the contamination proportion $\epsilon$, which is usually not available in practice. Additionally, the parameter $\epsilon$ is not identifiable in the setting of (\ref{eq:noise-sig}), and thus learning $\epsilon$ from data is impossible. This challenge of adaptive robust confidence interval construction has recently been studied by \cite{luo2024adaptive} in the simpler setting of a Gaussian location model. Given i.i.d.~observations generated from $(1-\epsilon)N(\beta,1)+\epsilon Q$ with some arbitrary $Q$, it is shown by \cite{luo2024adaptive} that the optimal length of a confidence interval covering the location parameter $\beta$ is of order $\frac{1}{\sqrt{\log n}}+\frac{1}{\sqrt{\log(1/\epsilon)}}$ when $\epsilon$ is unknown. In contrast, with the knowledge of $\epsilon$, the optimal length is of order $\frac{1}{\sqrt{n}}+\epsilon$. Therefore, ignorance of $\epsilon$ implies a significant adaptation cost in construction of robust confidence intervals. It is natural to suspect that such an adaptation cost is also present in the regression setting (\ref{eq:mul-lin-comp}).

A natural approach for handling unknown contamination levels is Lepski’s method \citep{lepski1997optimal,jain2022robust}, which has been successfully used for adaptive estimation. However, while such methods yield estimators that achieve optimal rates without knowing $\epsilon$, they do not directly provide confidence intervals with guaranteed coverage, as uncertainty quantification requires estimating the statistical error itself rather than merely selecting an estimator with optimal rate.

In this paper, we propose an algorithm that computes a confidence interval for a given coordinate of $b$. The procedure does not require knowledge of the contamination level $\epsilon$. We show that the resulting confidence interval achieves the desired coverage uniformly over all contamination distributions $Q$, and its length is bounded by $O_{\mathbb{P}}\left(\frac{1}{\sqrt{n(1-\epsilon)^2}}\right)$. In particular, this rate is minimax-optimal up to constants even when $\epsilon$ is known \citep{d2021consistent2}. Taken together, these results establish that one can simultaneously attain adaptivity to unknown contamination, minimax-optimal length, and robustness over arbitrary contamination distributions. In contrast to the negative result of \cite{luo2024adaptive}, this shows that there is no adaptation cost for robust confidence interval construction in the regression setting. The length of the confidence interval automatically shrinks as the data quality improves (i.e., as $\epsilon$ decreases). Moreover, the same theoretical guarantee continues to hold when the Gaussian noise $N(0,1)$ in (\ref{eq:noise-sig}) is replaced by other non-degenerate distributions such as $t$-distribution or the Cauchy distribution, demonstrating robustness to both heavy-tailed noise and arbitrary contamination.

Our algorithm is constructed from a one-dimensional estimating equation
\begin{equation}
\frac{1}{n}\sum_{i=1}^ng(\wt{x}_i\wh{b}_1-\wt{y}_i)\wt{x}_i=0, \label{eq:Z-intro}
\end{equation}
where
\begin{equation}
g(t)=\tanh(t)=\frac{e^t-e^{-t}}{e^t+e^{-t}}, \label{eq:tanh}
\end{equation}
and $\wt{x}_i,\wt{y}_i\in\mathbb{R}$ are transformations of $(X_i,y_i)$ such that $\wt{y}_i\approx \wt{x}_ib_1+z_i$ approximately holds. The details of constructing $(\wt{x}_i,\wt{y}_i)$ will be given in Algorithm \ref{alg:CI}. 
The interval estimator is then defined by
\begin{equation}
\left\{\beta: -\frac{C}{\sqrt{n}}\leq \frac{1}{n}\sum_{i=1}^ng(\wt{x}_i\beta-\wt{y}_i)\wt{x}_i\leq \frac{C}{\sqrt{n}}\right\}.\label{eq:CI-intro}
\end{equation}
In other words, instead of quantifying the uncertainty of a point estimator, the interval (\ref{eq:CI-intro}) directly quantifies the uncertainty of a Z-estimation objective. This powerful idea was recently proposed by \cite{chang2024confidence}. In the setting of robust regression, we are able to characterize the magnitude of $\frac{1}{n}\sum_{i=1}^ng(\wt{x}_ib_1-\wt{y}_i)\wt{x}_i$ uniformly over all contamination distributions $Q$ in (\ref{eq:noise-sig}). Moreover, the choice of $C$ in (\ref{eq:CI-intro}) does not depend on $\epsilon$, so that (\ref{eq:CI-intro}) is adaptive to the unknown contamination distribution and the unknown contamination proportion.

The Z-estimation objective (\ref{eq:Z-intro}) uses the hyperbolic tangent function, which is analytically smooth. In fact, for the purpose of point estimation, replacing $g(\cdot)$ by the derivative of Huber loss would also work. However, by using the smoother hyperbolic tangent, we are able to control the model misspecification error in the approximation $\wt{y}_i\approx \wt{x}_ib_1+z_i$ through a Taylor expansion.

It is interesting to note that (\ref{eq:CI-intro}) can also be derived from the general strategy of constructing adaptive robust confidence interval proposed by \cite{luo2024adaptive}. Specifically, as long as one can construct a testing function that is able to distinguish between different values of the parameter together with different contamination proportions, the testing function can be inverted into a robust confidence interval that is adaptive to the unknown contamination level $\epsilon$. Although they only stated this framework for location families, it turns out that the same principle can be applied to the regression setting. In particular, one can construct an appropriate test whose inversion recovers the formula (\ref{eq:CI-intro}). We refer the reader to Section \ref{sec:test-inv} for more details.


\subsection{Related Work}

Due to the lack of assumptions on contamination, confidence interval construction under the noise (\ref{eq:noise-sig}) is required to be \textit{nonparametric}, since there is no asymptotic distribution available. One class of nonparametric methods is based on the bootstrap. In particular, bootstrap in linear regression was developed in \cite{freedman1981bootstrapping}, and its variants and theoretical properties were further investigated by \cite{bickel1983bootstrapping,mammen1989asymptotics,mammen1993bootstrap}. However, bootstrap consistency is usually based on the existence of an asymptotic distribution, which is still violated by (\ref{eq:noise-sig}).

There are several recently proposed methods that impose very few assumptions on the noise variables. The cyclic permutation test \citep{lei2021assumption} and the residual permutation test \citep{wen2025residual} only assume the noise variables to be exchangeable, and their tests can be inverted into confidence intervals. However, the theoretical properties under (\ref{eq:noise-sig}) are unknown, and we find the two methods to be quite conservative in numerical experiments. The HulC procedure \citep{kuchibhotla2024hulc} is a general black-box tool to quantify uncertainty of a nearly median-unbiased point estimator. In practice, a large sample size is needed for the confidence interval not to be too wide due to the use of sample splitting in its implementation. Finally, robust universal inference \citep{park2023robust} is a general framework of constructing robust confidence intervals using sample splitting. Applying this framework to a specific model requires the design of a specific testing procedure. Whether this could be done for a linear model with noise of the form (\ref{eq:noise-sig}) is an interesting question to be explored.

\subsection{Paper Organization}

We introduce the new algorithm and state its theoretical guarantee in Section \ref{sec:main}. To understand why the proposed method works, we will first discuss the problem and illustrate the main idea for $p=1$ in Section \ref{sec:CI}. Then in Section \ref{sec:decorr}, we will discuss how to reduce the original problem to the one-dimensional setting through decorrelation. Extensive numerical experiments will be reported in Section \ref{sec:num}. We discuss two robust regression settings beyond the oblivious contamination model in Section \ref{sec:beyond} and collect all technical proofs in Section \ref{sec:proof}.

\subsection{Notation}

Define $[n] = \{1,\ldots,n\}$ for any positive integer $n$. For any two sequences $\{a_n\}$ and $\{b_n\}$, we write $a_n \asymp b_n$ if there exist constants $c, C>0$ such that $ca_n \leq b_n\leq Ca_n$ for all $n$; $a_n \lesssim b_n$ means that $a_n \leq C b_n$ holds for some constant $C > 0$ independent of $n$. For a vector $v$, its $\ell_2$ norm is defined by $\|v\|=\sqrt{\sum_iv_i^2}$. The $p$-dimensional unit sphere is defined by $S_{p-1}=\{v\in\mathbb{R}^p:\|v\|=1\}$. For a matrix $A$, its operator norm $\opnorm{A}$ is defined by the largest singular value. For a set $B$, we use $|B|$ for its cardinality. The CDF of $N(0,1)$ is denoted by $\Phi(\cdot)$. We use $\mathbb{E}$ and $\mathbb{P}$ for generic expectation and probability operators whenever the distribution is clear from the context.

\section{Main Results}\label{sec:main}

With $X_i=(x_i,w_i^T)^T$ and $b=(\beta,\theta^T)^T$, we rewrite the linear model (\ref{eq:mul-lin-comp}) as
\begin{equation}
y_i=\beta x_i + \theta^Tw_i + z_i, \label{eq:multiv-lin}
\end{equation}
where $\beta\in\mathbb{R}$ is the parameter of interest and $\theta\in\mathbb{R}^{p-1}$ collects all other regression coefficients. Define the Huber loss $\rho(\cdot)$ as
$$\rho(t)=\begin{cases}
t^2 & |t|\leq 1 \\
2|t|-1 & |t|>1.
\end{cases}$$
An algorithm of computing a robust confidence interval of $\beta$ is given below.
\begin{algorithm}
\DontPrintSemicolon
\SetKwInOut{Input}{Input}\SetKwInOut{Output}{Output}
\Input{$\{(x_i,w_i,y_i)\}_{i=1}^n$}
\Output{$\wh{\beta}_L$ and $\wh{\beta}_R$} 
\nl $\wh{\alpha} \leftarrow \argmin_{\alpha}\frac{1}{n}\sum_{i=1}^n(x_i-\alpha^Tw_i)^2$;

\nl For $i\in[n]$,

\qquad $\wt{x}_i\leftarrow x_i-\wh{\alpha}^Tw_i$;

\nl $\wh{\gamma} \leftarrow \argmin_{\gamma}\min_{\beta}\frac{1}{n}\sum_{i=1}^n\rho(y_i-\beta \wt{x}_i-\gamma^Tw_i)$;

\nl For $i\in[n]$,

\qquad $\wt{y}_i\leftarrow y_i-\wh{\gamma}^Tw_i$;

\nl $\wh{\beta}_L\leftarrow \inf\left\{\beta: \frac{\sum_{i=1}^ng\left(\wt{x}_i\beta-\wt{y}_i\right)\wt{x}_i}{\sqrt{\sum_{i=1}^n\wt{x}_i^2}}\geq -1.5\sqrt{\log(2/\alpha)}\right\}$;

$\wh{\beta}_R\leftarrow \sup\left\{\beta: \frac{\sum_{i=1}^ng\left(\wt{x}_i\beta-\wt{y}_i\right)\wt{x}_i}{\sqrt{\sum_{i=1}^n\wt{x}_i^2}}\leq 1.5\sqrt{\log(2/\alpha)}\right\}$.

\caption{Constructing Confidence Interval for $\beta$}
\label{alg:CI}
\end{algorithm}
Algorithm \ref{alg:CI} first computes $\{(\wt{x}_i,\wt{y}_i)\}_{i=1}^n$ by regressing $x_i$ and $y_i$ on $w_i$ respectively. Then, $\wt{x}_i$ and $\wt{y}_i$ are defined as the residuals. Since the covariates $\{(x_i,w_i)\}_{i=1}^n$ do not have contamination, regressing $x_i$ on $w_i$ can be done by ordinary least squares. On the other hand, regressing $y_i$ on $w_i$ involves the Huber loss given the presence of outliers in $\{y_i\}_{i=1}^n$. Once we obtain $\{(\wt{x}_i,\wt{y}_i)\}_{i=1}^n$, we apply the formula (\ref{eq:CI-intro}) with $C=1.5\sqrt{\frac{\log(2/\alpha)}{n}\sum_{i=1}^n\wt{x}_i^2}$ to obtain the interval $[\wh{\beta}_L,\wh{\beta}_R]$.

To state the theoretical guarantee of Algorithm \ref{alg:CI}, we impose the following assumption on the random design.
\begin{Assumption} \label{as:bounded moments}
There exist constants $K_1,K_2,K_3>0$, such that for any $v\in S_{p-1}$, the random vector $X_i=(x_i,w_i^T)^T\in\mathbb{R}^p$ satisfies
\begin{enumerate}
\item $\mathbb{P}\left(v^TX_i>t\right)=\mathbb{P}\left(v^TX_i<-t\right)\leq \exp\left(-\frac{t^2}{K_1^2}\right)$ for all $t>0$.
\item $\mathbb{E}|v^TX_i|^2\indi\{|v^TX_i|\leq K_2\}\geq K_3^{-1}$.
\end{enumerate}
\end{Assumption}
In words, for each direction $v\in S_{p-1}$, the projection $v^TX_i$ is a symmetric sub-Gaussian random variable whose distribution is not degenerate in a neighborhood of zero. Note that the gradient of the Huber loss $\rho(t)$ is non-constant only when $|t|$ is small, which is a crucial property for robustness against large outliers. The non-degeneracy in a neighborhood of zero is thus necessary for regression with Huber loss.

\begin{Theorem}\label{thm:main}
Consider i.i.d. samples $\{(x_i,w_i,y_i)\}_{i=1}^n$ generated from the linear model (\ref{eq:multiv-lin}), with $\{(x_i,w_i)\}_{i=1}^n$ and $\{z_i\}_{i=1}^n$ independent from each other and satisfying \prettyref{as:bounded moments} and (\ref{eq:noise-sig}). For any $\alpha\in(0,1)$, there exist constants $c>0$ and $C>0$ depending on $\alpha$, such that
the interval $[\wh{\beta}_L,\wh{\beta}_R]$ computed by Algorithm \ref{alg:CI} satisfies
$$\mathbb{P}\left(\beta\in [\wh{\beta}_L,\wh{\beta}_R]\text{ and }\wh{\beta}_R-\wh{\beta}_L\leq \frac{C}{\sqrt{n(1-\epsilon)^2}}\right) \geq 1-\alpha,$$
as long as $\frac{p^2}{n(1-\epsilon)^4}\leq c$.
\end{Theorem}

\begin{Remark}\label{rmk:noise}
The noise condition (\ref{eq:noise-sig}) assumes that $z_i$'s are generated from a standard $N(0,1)$ if they are not drawn from contamination. Our proofs would remain the same if $N(0,1)$ is replaced by any distribution satisfying $\mathbb{P}(|z_i|\leq 1)\gtrsim 1$, which includes most heavy tailed distributions such as $t$ and Cauchy. More generally, if the noise is scaled by $\sigma$ (e.g. $N(0,\sigma^2)$), we can scale Algorithm \ref{alg:CI} accordingly by $g(\cdot)\rightarrow g(\cdot/\tau)$ and $\rho(\cdot)\rightarrow \rho(\cdot/\tau)$ for some $\tau>0$. Then, Theorem \ref{thm:main} continues to hold as long as $\tau\lesssim \sigma$ and $\frac{p^2}{n(1-\epsilon)^4(\tau/\sigma)^4}$ is sufficiently small.
\end{Remark}

\begin{Remark}\label{rmk:sub-G}
The sub-Gaussianity of the design can be relaxed to weaker conditions, and the result would still hold under stronger conditions on dimension. For example, under the sub-Weibull tail $\exp\left(-(t/K_1)^{\kappa}\right)$ for some $\kappa\in(0,2]$, the condition $\frac{p^2}{n(1-\epsilon)^4}\leq c$ would be generalized to $\frac{p^{1+\frac{2}{\kappa}}}{n(1-\epsilon)^4}\leq c$. The same proof technique still works after incorporating concentration inequalities in \cite{Kuchibhotla_2022}.
\end{Remark}

\begin{Remark}\label{rmk:exp-inf}
It is crucial that the interval length guarantee holds with high probability instead of in expectation. It can be shown that $\mathbb{E}|\wh{\beta}_R-\wh{\beta}_L|=\infty$ (see Proposition \ref{prop:inf-exp-len}) in the worst-case. This is because with probability $\epsilon^n$, all $z_i$'s are drawn from $Q$ under (\ref{eq:noise-sig}), and thus the expectation is at least $\epsilon^n$ times the expectation under $Q$, which can be arbitrarily large.
\end{Remark}

Theorem \ref{thm:main} also implies that the point estimator $\wh{\beta}=\frac{\wh{\beta}_L+\wh{\beta}_R}{2}$ achieves the error bound $O_{\mathbb{P}}\left(\frac{1}{\sqrt{n(1-\epsilon)^2}}\right)$, which is known to be optimal given the lower bound of estimation in \cite{d2021consistent2}. This immediately implies that the length of $[\wh{\beta}_L,\wh{\beta}_R]$ is also rate optimal, since a smaller interval length would imply an even better point estimator. Moreover, Algorithm \ref{alg:CI} does not require the knowledge of $\epsilon$, and thus it is adaptive to the unknown contamination proportion. Since the lower bound of \cite{d2021consistent2} holds even if the value of $\epsilon$ is given, Theorem \ref{thm:main} implies that there is no adaptation cost in terms of the rate. Compared with the necessity of adaptation cost under the Gaussian location model \citep{luo2024adaptive}, the result of Theorem \ref{thm:main} is unexpected.

\prettyref{as:bounded moments} on the covariate assumes symmetry. The symmetry condition can be avoided by incorporating a preprocessing step of computing the pairwise difference $y_i-y_j=(X_i-X_j)^T\beta+(z_i-z_j)$. The new covariate $X_i-X_j$ certainly satisfies symmetry regardless of the distribution of $X_i$. 


We finally give a brief discussion on the condition $\frac{p^2}{n(1-\epsilon)^4}\leq c$ for some sufficiently small constant $c>0$. When $\epsilon$ is bounded away from $1$, this condition allows $p$ to grow with $n$ as long as $n \gg p^2$. The requirement that the sample size needs to exceed the squared degrees of freedom has been previously noted in the high-dimensional literature when the signal is very sparse \citep{zhang2014confidence,javanmard2014confidence,van2014asymptotically}, and can be shown to be necessary for interval estimation with parametric rate \citep{ren2015asymptotic,cai2017confidence,verzelen2018adaptive}. 
When the signal is dense, however, the condition $n \gg p^2$ is not necessary for confidence interval construction in the absence of contamination \citep{portnoy1985asymptotic,wen2025residual}. In contrast, under Huber's contamination model \eqref{eq:noise-sig}, the optimal sample complexity for confidence interval construction remains unknown. In particular, it is unclear whether the quadratic dependence on $p$ is information-theoretically necessary, even without computational constraints, or whether it can be improved by alternative procedures. Establishing the exact dependence on $p$ and $\epsilon$, as well as understanding possible gaps between information-theoretic and computational limits, are important open problems for future work.

In the next two sections, we will introduce the idea behind each step of Algorithm \ref{alg:CI}. We will first consider a simpler problem of $p=1$ without any nuisance parameter in Section \ref{sec:CI}, and explain why the formula (\ref{eq:CI-intro}) has coverage property under (\ref{eq:noise-sig}) with arbitrary contamination. Section \ref{sec:decorr} will motivate the formulas of $\wt{x}_i$ and $\wt{y}_i$ and explain why Algorithm \ref{alg:CI} works in the original model (\ref{eq:multiv-lin}).

\section{Adaptive Confidence Interval for Univariate Regression}\label{sec:CI}

In this section, we will introduce general methodology of confidence interval construction in a univariate version of (\ref{eq:multiv-lin}) without nuisance parameter. The extension to the multivariate setting will be considered in Section \ref{sec:decorr}.
To be specific, we consider i.i.d. samples $\{(x_i,y_i)\}_{i=1}^n$ that are generated by a linear model
\begin{equation}
y_i=\beta x_i+z_i,\label{eq:1dtoy}
\end{equation}
where $x_i\sim N(0,1)$ and $z_i\sim (1-\epsilon)N(0,1)+\epsilon Q$ are independent from each other. Our goal is to construct a confidence interval when the contamination proportion $\epsilon$ is unknown.

\subsection{A Symmetric Location Model}\label{sec:loc}

To motivate the proposed method for (\ref{eq:1dtoy}), we first study a very simple location model with symmetric error,
\begin{equation}
y_i=\beta+z_i,\label{eq:location}
\end{equation}
where $z_i\sim (1-\epsilon)N(0,1)+\epsilon Q_0$ for some distribution $Q_0$ that is symmetric around zero. In other words, $z_i\sim Q_0$ implies $-z_i\sim Q_0$. We emphasize that such a symmetry assumption is not assumed for the error distribution in (\ref{eq:1dtoy}).

Since $\beta$ is the center of a symmetric distribution regardless of what $Q_0$ is, a natural estimator for $\beta$ is the sample median $\wh{\beta}=\text{Median}(\{y_i\}_{i=1}^n)$, which achieves the error bound\footnote{The non-asymptotic bound follows the same analysis of sample median in \cite{chen2018robust}. The explicit constant $3.5$ can be obtained by assuming that $n(1-\epsilon)^2$ is sufficiently large. Moreover, the rate $\frac{1}{\sqrt{n(1-\epsilon)^2}}$ is minimax optimal given the existence of a symmetric distribution $Q_0$ such that $(1-\epsilon)N(0,1)+\epsilon Q_0=N(0,(1-\epsilon)^{-2})$ \citep{d2021consistent2}.}
$$\mathbb{P}\left(|\wh{\beta}-\beta|\leq 3.5\frac{1}{\sqrt{n(1-\epsilon)^2}}\right)\geq 0.95.$$
This high-probability error bound immediately implies that
\begin{equation}
\left[\wh{\beta}-3.5\frac{1}{\sqrt{n(1-\epsilon)^2}},\wh{\beta}+3.5\frac{1}{\sqrt{n(1-\epsilon)^2}}\right] \label{eq:median-eps}
\end{equation}
is a valid confidence interval that has coverage probability at least $0.95$. However, since (\ref{eq:median-eps}) depends on the knowledge of $\epsilon$, it cannot be used when $\epsilon$ is unknown.

For any $s\in[0,1]$, define the left and right empirical quantile functions by
\begin{eqnarray*}
q_n^-(s) &=& \inf\left\{t: \frac{1}{n}\sum_{i=1}^n\indi\{y_i\leq t\}\geq s\right\}, \\
q_n^+(s) &=& \sup\left\{t: \frac{1}{n}\sum_{i=1}^n\indi\{y_i< t\}\leq s\right\}.
\end{eqnarray*}
Then, the sample median can be written as $\wh{\beta}=\frac{q_n^-\left(\frac{1}{2}\right)+q_n^+\left(\frac{1}{2}\right)}{2}$. Instead of quantifying the uncertainty of the median as in (\ref{eq:median-eps}), we could directly quantify the uncertainty of the quantile level by considering the interval
\begin{equation}
\left[q_n^-\left(\frac{1}{2}-\sqrt{\frac{\log(2/\alpha)}{2n}}\right),q_n^+\left(\frac{1}{2}+\sqrt{\frac{\log(2/\alpha)}{2n}}\right)\right]. \label{eq:ci-quant}
\end{equation}
Interestingly, the interval (\ref{eq:ci-quant}) covers $\beta$ with probability at least $1-\alpha$ regardless of the value of $\epsilon$. Indeed, the event that the interval (\ref{eq:ci-quant}) covers the true $\beta$ in (\ref{eq:location}) is equivalent to
\begin{equation}
\frac{1}{n}\sum_{i=1}^n\indi\{z_i\leq 0\}\geq \frac{1}{2}-\sqrt{\frac{\log(2/\alpha)}{2n}}\quad\text{and}\quad\frac{1}{n}\sum_{i=1}^n\indi\{z_i<0\}\leq \frac{1}{2}+\sqrt{\frac{\log(2/\alpha)}{2n}}, \label{eq:coverage-event}
\end{equation}
which hold by Hoeffding's inequality using the symmetry of $z_i$. In fact, the random variable $\sum_{i=1}^n\indi\{z_i\leq 0\}$ is almost distribution free, since it follows $\text{Binomial}(n,1/2)$ unless $Q_0$ has a point mass at $0$. One can even replace the $\sqrt{\frac{\log(2/\alpha)}{2n}}$ in (\ref{eq:ci-quant}) with an appropriate quantity computed from the CDF of $\text{Binomial}(n,1/2)$ to achieve an almost exact coverage level $1-\alpha$.

The interval (\ref{eq:ci-quant}) was originally considered in \cite{scheffe1945non} as a nonparametric confidence interval for location. Though the definition of (\ref{eq:ci-quant}) does not depend on the knowledge of $\epsilon$, its length still adapts to the unknown $\epsilon$.

\begin{Proposition}\label{prop:loc-med}
Consider i.i.d. samples from the location model with error distribution $z_i\sim (1-\epsilon)N(0,1)+\epsilon Q_0$ for some distribution $Q_0$ that is symmetric around zero. 
For any $\alpha\in(0,1)$, there exist constants $c>0$ and $C>0$ depending on $\alpha$, such that the interval (\ref{eq:ci-quant}) covers $\beta$ with probability at least $1-\alpha$. Moreover, the length of the interval (\ref{eq:ci-quant}) is bounded by $\frac{C}{\sqrt{n(1-\epsilon)^2}}$ with probability at least $1-\alpha$, as long as $\frac{1}{\sqrt{n(1-\epsilon)^2}}\leq c$.
\end{Proposition}

\subsection{Confidence Interval for Z-Estimation}\label{sec:z}

Scheffe and Tukey's interval (\ref{eq:ci-quant}) is specifically designed for the inference of population quantiles. To extend the idea to the regression setting (\ref{eq:1dtoy}), it is helpful to understand (\ref{eq:ci-quant}) from a more general perspective.

To this end, we define
$$Z_n(\beta)=\frac{1}{n}\sum_{i=1}^n\text{sign}(\beta-y_i).$$
The sample median can also be regarded as a Z-estimator, since it is an approximate solution to the equation $Z_n(\beta)=0$. Using the objective function of Z-estimation, we can write (\ref{eq:ci-quant}) as
\begin{equation}
\left\{\beta: -\sqrt{\frac{2\log(2/\alpha)}{n}}\leq Z_n(\beta)\leq \sqrt{\frac{2\log(2/\alpha)}{n}}\right\}. \label{eq:Z-CI}
\end{equation}
An advantage of the formula (\ref{eq:Z-CI}) is that it is applicable in settings beyond the one-dimensional location model (\ref{eq:location}), in particular when quantile function is not available. For example, for the regression model (\ref{eq:1dtoy}), one can simply use the same formula (\ref{eq:Z-CI}) with $Z_n(\beta)$ chosen to be an appropriate score function for regression. Construction of general confidence sets based on Z-estimation has recently been considered by \cite{chang2024confidence} even when the parameter of interest is multivariate. It turns out that this is a correct framework for us to derive robust confidence intervals for (\ref{eq:1dtoy}) when $\epsilon$ is unknown.

\subsection{Application in Robust Regression}\label{sec:ratio}

Return to the one-dimensional linear model (\ref{eq:1dtoy}). We will construct a confidence interval for $\beta$ using ideas developed in Section \ref{sec:loc} and Section \ref{sec:z}.

One straightforward way of analyzing (\ref{eq:1dtoy}) is to consider the ratio statistic
\begin{equation}
\frac{y_i}{x_i}=\beta+\frac{z_i}{x_i}. \label{eq:ratio}
\end{equation}
Since the covariate $x_i$ follows $N(0,1)$ and the noise variable $z_i$ is independent from $x_i$, the random variable $\frac{z_i}{x_i}$ is symmetric around zero regardless of the noise distribution. Therefore, (\ref{eq:ratio}) is an instance of the symmetric location model (\ref{eq:1dtoy}), where $\beta$ is the center of the distribution. Thus, the regression coefficient $\beta$ is identifiable even when $\epsilon=1$. When the distribution of $x_i$ is not necessarily $N(0,1)$ or not even symmetric, one could still deal with a linear model with symmetric design by preprocessing the data with pairwise difference $y_i-y_j=\beta(x_i-x_j)+z_i-z_j$.

With the reduction to a symmetric location model, the confidence interval (\ref{eq:ci-quant}) can be applied with quantile functions computed from the ratio statistics $\left\{y_i/x_i\right\}_{i=1}^n$. Similar to Proposition \ref{prop:loc-med}, the coverage and length properties of this construction continue to hold for the regression model (\ref{eq:1dtoy}). Unfortunately, when it comes to the multivariate setting (\ref{eq:multiv-lin}) with the presence of nuisance, the idea of taking ratio does not apply.

On the other hand, the framework of Z-estimation has more flexibility. A natural robust estimator for (\ref{eq:1dtoy}) is median regression which minimizes
\begin{equation}
\frac{1}{n}\sum_{i=1}^n|y-x_i\beta|.\label{eq:med-reg}
\end{equation}
The gradient of (\ref{eq:med-reg}) is given by
\begin{equation}
Z_n(\beta)=\frac{1}{n}\sum_{i=1}^n\text{sign}(x_i\beta-y_i)x_i.\label{eq:med-reg-g}
\end{equation}
Median regression can also be defined as a Z-estimator that solves the equation $Z_n(\beta)=0$. Similar to (\ref{eq:Z-CI}), one can use the interval
\begin{equation}
\left\{\beta: -\frac{C}{\sqrt{n}}\leq Z_n(\beta)\leq \frac{C}{\sqrt{n}}\right\}, \label{eq:Z-CI-med-reg}
\end{equation}
with the function (\ref{eq:med-reg-g}) for some constant $C$ that depends on the coverage probability. To choose $C$, we note that the event that the interval (\ref{eq:Z-CI-med-reg}) contains the true $\beta$ is equivalent to
$$\left|\frac{1}{n}\sum_{i=1}^n\text{sign}(z_i)x_i\right|\leq \frac{C}{\sqrt{n}}.$$
Since $\{x_i\}$ and $\{z_i\}$ are independent, the random variable $\frac{1}{n}\sum_{i=1}^n\text{sign}(z_i)x_i$ is stochastically dominated by $N(0,n^{-1})$, and thus the Gaussian quantile $C=\Phi^{-1}(1-\alpha/2)$ guarantees coverage. More generally, one could also take $C=\sqrt{\frac{2}{n}\sum_{i=1}^nx_i^2\log(2/\alpha)}$ so that the coverage guarantee holds beyond the Gaussian design by applying Hoeffding's inequality to the self-normalizing sum. 

\subsection{A Smooth Objective}

For the general multivariate linear model (\ref{eq:multiv-lin}),
in order to estimate $\beta$ or construct its confidence interval, our strategy is to reduce (\ref{eq:multiv-lin}) to the one-dimensional setting (\ref{eq:1dtoy}). This is implemented through a decorrelation technique to be introduced in Section \ref{sec:decorr}. Our technique will construct pairs $\{(\wt{x}_i,\wt{y}_i)\}_{i=1}^n$ from (\ref{eq:multiv-lin}) such that
\begin{equation}
\wt{y}_i\approx \beta \wt{x}_i+z_i.\label{eq:appr-lin}
\end{equation}
In other words, the pairs $\{(\wt{x}_i,\wt{y}_i)\}_{i=1}^n$ follow the linear model (\ref{eq:1dtoy}) approximately. When applying (\ref{eq:Z-CI-med-reg}) with $\{(\wt{x}_i,\wt{y}_i)\}_{i=1}^n$, one also needs to account for the model misspecification error in (\ref{eq:appr-lin}) for the coverage property. With the Z-estimation objective given by (\ref{eq:med-reg-g}), the effect of model misspecification error is very hard to control due to the nonsmoothness of $\text{sign}(\cdot)$.

To overcome this difficulty, we consider a smooth version of (\ref{eq:med-reg-g}), given by
\begin{equation}
\frac{1}{n}\sum_{i=1}^ng(x_i\beta-y_i)x_i, \label{eq:smooth-z}
\end{equation}
where $g(\cdot)$ is the hyperbolic tangent function defined by (\ref{eq:tanh}). We note that $g(\cdot)$ can be regarded as a smoothed sign function. The corresponding M-estimator of (\ref{eq:smooth-z}) minimizes
$$\frac{1}{n}\sum_{i=1}^n\log\cosh(y_i-x_i\beta),$$
where $\cosh(t)=\frac{e^t+e^{-t}}{2}$. The function $\log\cosh(\cdot)$ is very similar to Huber loss since $\log\cosh(t)\sim \frac{t^2}{2}$ as $t\rightarrow 0$ and $\log\cosh(t)\sim |t|$ as $t\rightarrow\infty$. Unlike Huber loss, $\log\cosh(\cdot)$ has bounded derivatives of all orders. Thus, when applying (\ref{eq:smooth-z}) to (\ref{eq:appr-lin}), the effect of model misspecification can be controlled by Taylor expansion.

\begin{Proposition}\label{prop:cov-gaus}
Consider i.i.d. samples $\{(x_i,y_i)\}_{i=1}^n$ generated from the linear model (\ref{eq:1dtoy}), with $x_i\sim N(0,1)$ and $z_i\sim (1-\epsilon)N(0,1)+\epsilon Q$ independent from each other. For any $\alpha\in(0,1)$, there exist constants $c>0$ and $C>0$ depending on $\alpha$, such that the interval
\begin{equation}
\left\{\beta: -\sqrt{\frac{2\log(2/\alpha)}{n}}\leq \frac{\frac{1}{n}\sum_{i=1}^ng(x_i\beta-y_i)x_i}{\sqrt{\frac{1}{n}\sum_{i=1}^nx_i^2}}\leq \sqrt{\frac{2\log(2/\alpha)}{n}}\right\},\label{eq:Z-CI-med-reg-tau}
\end{equation}
covers $\beta$ with probability at least $1-\alpha$. Moreover, the length of the interval (\ref{eq:Z-CI-med-reg-tau}) is bounded by $\frac{C}{\sqrt{n(1-\epsilon)^2}}$ with probability at least $1-\alpha$, as long as $\frac{1}{\sqrt{n(1-\epsilon)^2}}\leq c$.
\end{Proposition}

\subsection{Why is Adaptation Possible?}\label{sec:test-inv}

The paper \cite{luo2024adaptive} considers confidence interval construction for
\begin{equation}
y_1,\cdots, y_n\overset{iid}\sim (1-\epsilon)N(\beta,1)+\epsilon Q, \label{eq:huber-mean}
\end{equation}
where $Q$ is some arbitrary contamination distribution. This model can be regarded as (\ref{eq:location}) without the symmetry condition on the error distribution. While the optimal length of confidence interval is of order $\frac{1}{\sqrt{n}}+\epsilon$ when $\epsilon$ is known, \cite{luo2024adaptive} proves that the optimal length slows down to $\frac{1}{\sqrt{\log(n)}}+\frac{1}{\sqrt{\log(1/\epsilon)}}$ when $\epsilon$ is unknown.

The adaptation cost caused by unknown $\epsilon$ in the setting of (\ref{eq:huber-mean}) can be explained by the difficulty of solving the following hypothesis testing problem,
\begin{equation}\label{eq:test-loc}
	\begin{split} &\begin{array}{l}
		H_{0}: y_1,\cdots, y_n \overset{iid}\sim P\in\left\{(1-\epsilon_{\max})N(\beta,1)+\epsilon_{\max} Q: Q\right\},
	\end{array}\\
	&\begin{array}{l}
		H_{1}: y_1,\cdots, y_n \overset{iid}\sim P\in\left\{(1-\epsilon)N(\beta+r,1)+\epsilon Q: Q\right\},
	\end{array}
	\end{split}
\end{equation}
where $\epsilon_{\max}$ is some constant that is strictly smaller than $1/2$. It was shown by \cite{luo2024adaptive} that an adaptive confidence interval under (\ref{eq:huber-mean}) directly leads to a solution to (\ref{eq:test-loc}). To be specific, one can reject the null whenever $\beta$ is not contained in the interval. Then, as long as the parameter $r$ under the alternative is greater than the interval length, both Type-1 and Type-2 errors can be controlled. Consequently, the smallest $r$ such that (\ref{eq:test-loc}) can be solved serves as a lower bound for the length of adaptive confidence interval under (\ref{eq:huber-mean}). When $r\lesssim \frac{1}{\sqrt{\log(n)}}+\frac{1}{\sqrt{\log(1/\epsilon)}}$, \cite{luo2024adaptive} showed that there exist $Q_0$ and $Q_1$ such that
$$(1-\epsilon_{\max})N(\beta,1)+\epsilon_{\max} Q_0\approx (1-\epsilon)N(\beta+r,1)+\epsilon Q_1$$
in the sense that the total variation distance between the two product distributions is sufficiently small; it is impossible to distinguish null from alternative.

Analogously, there is also a hypothesis testing problem associated with the construction of adaptive confidence intervals under (\ref{eq:1dtoy}). Let us use $P_{\epsilon,\beta,Q}$ for the joint distribution of $(x_i,y_i)$ under the linear model (\ref{eq:1dtoy}). Recall that $x_i\sim N(0,1)$ and $z_i\sim (1-\epsilon)N(0,1)+\epsilon Q$, and they are independent from each other. The testing problem is
\begin{equation}\label{eq:test-reg}
	\begin{split} &\begin{array}{l}
		H_{0}: (x_1,y_1),\cdots, (x_n,y_n) \overset{iid}\sim P\in\left\{P_{1,\beta,Q}:Q\right\},
	\end{array}\\
	&\begin{array}{l}
		H_{1}: (x_1,y_1),\cdots, (x_n,y_n) \overset{iid}\sim P\in\left\{P_{\epsilon,\beta+r,Q}:Q\right\}.
	\end{array}
	\end{split}
\end{equation}
Here, we can take $\epsilon_{\max}=1$ since $\beta$ is identifiable in (\ref{eq:1dtoy}). Unlike (\ref{eq:test-loc}), there exists a testing function for (\ref{eq:test-reg}) with small Type-1 and Type-2 error as long as $r\gtrsim \frac{1}{\sqrt{n(1-\epsilon)^2}}$, which exactly matches the minimax rate of estimating $\beta$ under (\ref{eq:1dtoy}).

It turns out that an informative testing statistic is the function $Z_n(\beta)=\frac{1}{n}\sum_{i=1}^n\text{sign}(x_i\beta-y_i)x_i$ in (\ref{eq:med-reg-g}). Under null, $Z_n(\beta)=-\frac{1}{n}\sum_{i=1}^n\text{sign}(z_i)x_i$ so that it has zero expectation. Under alternative, $Z_n(\beta)=-\frac{1}{n}\sum_{i=1}^n\text{sign}(x_ir+z_i)x_i$, where $z_i\sim (1-\epsilon)N(0,1)+\epsilon Q$. Thus, the expectation of $Z_n(\beta)$ under alternative is at most
$$-(1-\epsilon)\mathbb{E}_{x,z\overset{iid}\sim N(0,1)}\left[\text{sign}(xr+z)x\right]\lesssim -(1-\epsilon)r.$$
Then, by rejecting the null whenever $Z_n(\beta)< -\frac{C}{\sqrt{n}}$, one can solve the testing problem (\ref{eq:test-reg}) as long as the difference between means under null and alternative exceeds the noise level, which is $(1-\epsilon)r\gtrsim \frac{1}{\sqrt{n}}$, equivalent to the condition $r\gtrsim \frac{1}{\sqrt{n(1-\epsilon)^2}}$. Similarly, the testing statistic (\ref{eq:smooth-z}) also works under the same condition.

\cite{luo2024adaptive} also shows that any testing procedure solving (\ref{eq:test-reg}) can be inverted into an adaptive confidence interval for (\ref{eq:1dtoy}) by collecting all $\beta$ such that the null hypothesis in (\ref{eq:test-reg}) is not rejected. The rejection region $Z_n(\beta)< -\frac{C}{\sqrt{n}}$ then immediately leads to the one-sided interval $\left\{\beta: Z_n(\beta)\geq -\frac{C}{\sqrt{n}}\right\}$. By replacing the $\beta+r$ in the alternative of (\ref{eq:test-reg}) with $\beta-r$, one can also obtain $\left\{\beta: Z_n(\beta)\leq \frac{C}{\sqrt{n}}\right\}$ in a similar way. The intersection of the two intervals recovers (\ref{eq:Z-CI-med-reg}). Using (\ref{eq:smooth-z}) instead of $Z_n(\beta)$ also recovers (\ref{eq:Z-CI-med-reg-tau}).

\section{Extension to Multivariate Regression}\label{sec:decorr}

For the general multivariate setting (\ref{eq:multiv-lin}), our main idea to construct an adaptive confidence interval for $\beta$ is a reduction from (\ref{eq:multiv-lin}) to (\ref{eq:1dtoy}) so that the methodology introduced in Section \ref{sec:CI} can be applied.

\subsection{Decorrelating Nuisance Parameters}

One naive way of regarding (\ref{eq:multiv-lin}) as a univariate model is to treat $\theta^Tw_i+z_i$ as a noise variable and directly apply (\ref{eq:Z-CI-med-reg-tau}) to the pairs $\{(x_i,y_i)\}_{i=1}^n$. However, there are two problems with this approach. First, $\theta^Tw_i+z_i$ is not independent from $x_i$. Second, the variance of $\theta^Tw_i+z_i$ when $z_i$ is not contaminated is $O(\|\theta\|^2+\sigma^2)$, which can be arbitrarily large.

To address these two issues, Algorithm \ref{alg:CI} considers the transformation
\begin{eqnarray}
\label{eq:x-tilde} \wt{x}_i &=& x_i-\wh{\alpha}^Tw_i, \\
\label{eq:y-tilde} \wt{y}_i &=& y_i-\wh{\gamma}^Tw_i,
\end{eqnarray}
so that (\ref{eq:multiv-lin}) can be equivalently written as
\begin{equation}
\wt{y}_i = \beta\wt{x}_i + (\gamma-\wh{\gamma})^Tw_i + z_i, \label{eq:alg-rearran}
\end{equation}
where $\gamma$ stands for the vector $\beta\wh{\alpha}+\theta$. Since $\wh{\alpha}$ in (\ref{eq:x-tilde}) is the least-squares estimator computed from regressing $\{x_i\}_{i=1}^n$ on $\{w_i\}_{i=1}^n$, it leads to the following identity
\begin{equation}
\frac{1}{n}\sum_{i=1}^n\wt{x}_iw_i=0. \label{eq:sample-uncorr}
\end{equation}
Even though $\wt{x}_i$ and $(\gamma-\wh{\gamma})^Tw_i+z_i$ are still not independent in (\ref{eq:alg-rearran}), the independence between $\wt{x}_i$ and $z_i$ and the uncorrelatedness between $\wt{x}_i$ and $w_i$ are sufficient for us to apply (\ref{eq:Z-CI-med-reg-tau}). We emphasize that the uncorrelatedness condition (\ref{eq:sample-uncorr}) is only in the sample level, which turns out to be particularly convenient when bounding the model misspecification error. Moreover, the $\wh{\gamma}$ in (\ref{eq:y-tilde}) is a robust estimator for $\gamma=\beta\wh{\alpha}+\theta$, which will ensure that the scale of the noise variable $(\gamma-\wh{\gamma})^Tw_i+z_i$ is well controlled.

Applying the formula (\ref{eq:smooth-z}) to $\{(\wt{x}_i,\wt{y}_i)\}_{i=1}^n$, we define
\begin{equation}
G_n(\beta)=\frac{1}{n}\sum_{i=1}^ng\left(\wt{x}_i\beta-\wt{y}_i\right)\wt{x}_i. \label{eq:G-y-x-tile}
\end{equation}
Similar to (\ref{eq:Z-CI-med-reg-tau}), we use the confidence interval
\begin{equation}
\left\{\beta: -1.5\sqrt{\frac{\log(2/\alpha)}{n}}\leq \frac{G_n(\beta)}{\sqrt{\frac{1}{n}\sum_{i=1}^n\wt{x}_i^2}}\leq 1.5\sqrt{\frac{\log(2/\alpha)}{n}}\right\}, \label{eq:main-CI}
\end{equation}
which is exactly the interval $[\wh{\beta}_L,\wh{\beta}_R]$ computed by Algorithm \ref{alg:CI}. The constant $1.5$ in (\ref{eq:main-CI}) is chosen to be slightly greater than the $\sqrt{2}$ used in (\ref{eq:Z-CI-med-reg-tau}) to account for the model misspecification caused by $(\gamma-\wh{\gamma})^Tw_i$.

\subsection{Coverage}\label{sec:cov}

In order that the interval (\ref{eq:main-CI}) still has the coverage property, we need to understand the behavior of $G_n(\beta)$ when $\beta$ is the true parameter that generates the data. Plugging (\ref{eq:alg-rearran}) into (\ref{eq:G-y-x-tile}), we obtain
$$G_n(\beta)=-\frac{1}{n}\sum_{i=1}^ng\left((\gamma-\wh{\gamma})^Tw_i + z_i\right)\wt{x}_i.$$
To analyze the effect of model misspecification caused by $(\gamma-\wh{\gamma})^Tw_i$, we use the smoothness of $g(\cdot)$ and have
\begin{equation}
-G_n(\beta) - \frac{1}{n}\sum_{i=1}^ng\left(z_i\right)\wt{x}_i = \frac{1}{n}\sum_{i=1}^ng'\left(z_i\right)(\gamma-\wh{\gamma})^Tw_i\wt{x}_i + \frac{1}{2n}\sum_{i=1}^ng''(\xi_i)|(\gamma-\wh{\gamma})^Tw_i|^2\wt{x}_i, \label{eq:taylor}
\end{equation}
where $\xi_i$ is some number between $(\gamma-\wh{\gamma})^Tw_i + z_i$ and $z_i$. We will separately bound the two terms on the right hand side of (\ref{eq:taylor}).

For the first term, we have
\begin{eqnarray}
\nonumber && \left|\frac{1}{n}\sum_{i=1}^ng'\left(z_i\right)(\gamma-\wh{\gamma})^Tw_i\wt{x}_i\right| \\
\label{eq:use-uncorr} &=& \left|\frac{1}{n}\sum_{i=1}^n\left(g'\left(z_i\right)-\mathbb{E}g'\left(z_i\right)\right)(\gamma-\wh{\gamma})^Tw_i\wt{x}_i\right| \\
\nonumber &\leq& \|\wh{\gamma}-\gamma\|\left\|\frac{1}{n}\sum_{i=1}^n\left(g'\left(z_i\right)-\mathbb{E}g'\left(z_i\right)\right)\wt{x}_iw_i\right\| \\
\label{eq:l2-norm} &\lesssim& \|\wh{\gamma}-\gamma\|\sqrt{\frac{p}{n}},
\end{eqnarray}
where the last bound holds with high probability. The equality (\ref{eq:use-uncorr}) uses the identity $\frac{1}{n}\sum_{i=1}^n\left(\mathbb{E}g'\left(z_i\right)\right)\wt{x}_iw_i=0$, which is by (\ref{eq:sample-uncorr}) and the fact that $\{z_i\}_{i=1}^n$ are identically distributed. The last step (\ref{eq:l2-norm}) is by Lemma \ref{lem:for-main}, which bounds the $\ell_2$ norm of a $p$-dimensional vector by $O_{\mathbb{P}}\left(\sqrt{\frac{p}{n}}\right)$ since each of its entry is of order $O_{\mathbb{P}}(n^{-1/2})$.

For the second term, we have
\begin{eqnarray}
\nonumber && \left|\frac{1}{2n}\sum_{i=1}^ng''(\xi_i)|(\gamma-\wh{\gamma})^Tw_i|^2\wt{x}_i\right| \\
\label{eq:g''} &\leq& \frac{1}{2n}\sum_{i=1}^n |(\gamma-\wh{\gamma})^Tw_i|^2|\wt{x}_i| \\
\nonumber &\leq& \frac{\|\wh{\gamma}-\gamma\|^2}{2}\opnorm{\frac{1}{n}\sum_{i=1}^n|\wt{x}_i|w_iw_i^T} \\
\label{eq:opbd} &\lesssim& \|\wh{\gamma}-\gamma\|^2,
\end{eqnarray}
with high probability. The inequality (\ref{eq:g''}) is by the fact $\|g''\|_{\infty}\leq 1$ and the operator norm bound used in  (\ref{eq:opbd}) is by Lemma \ref{lem:for-main}.

Combining the bounds for the two terms on the right hand side of (\ref{eq:taylor}), we can write down the following result.
\begin{Lemma}\label{lem:decor}
Under the setting of Theorem \ref{thm:main}, for any constant $\delta\in(0,1)$, there exist some constants $c>0$ and $C>0$ depending on $\delta$ such that
$$|G_n(\beta)|\leq\left|\frac{1}{n}\sum_{i=1}^ng\left(z_i\right)\wt{x}_i\right|+ C\left(\|\wh{\gamma}-\gamma\|\sqrt{\frac{p}{n}} + \|\wh{\gamma}-\gamma\|^2\right),$$
with probability at least $1-\delta$ as long as $\frac{p^2}{n(1-\epsilon)^4}\leq c$.
\end{Lemma}

Therefore, to guarantee coverage, it suffices to show that
\begin{equation}
\left|\frac{1}{n}\sum_{i=1}^ng\left(z_i\right)\wt{x}_i\right|\leq 1.45\sqrt{\frac{1}{n}\sum_{i=1}^n\wt{x}_i^2}\sqrt{\frac{\log(2/\alpha)}{n}}, \label{eq:cov-general}
\end{equation}
and
\begin{equation}
C\left(\|\wh{\gamma}-\gamma\|\sqrt{\frac{p}{n}} + \|\wh{\gamma}-\gamma\|^2\right) \leq \frac{1}{20}\sqrt{\frac{1}{n}\sum_{i=1}^n\wt{x}_i^2}\sqrt{\frac{\log(2/\alpha)}{n}}. \label{eq:bias-control-est}
\end{equation}
The first bound (\ref{eq:cov-general}) can be established in a similar way to Proposition \ref{prop:cov-gaus} using concentration of self-normalizing sums. The second bound (\ref{eq:bias-control-est}) is implied by the following estimation error bound.
\begin{Lemma}\label{lem:estimation-huber}
Under the setting of Theorem \ref{thm:main}, for any constant $\delta\in(0,1)$, there exist some constants $c>0$ and $C>0$ depending on $\delta$ such that
$$\|\wh{\gamma}-\gamma\|^2\leq  C\frac{p}{n(1-\epsilon)^2},$$
with probability at least $1-\delta$ as long as $\frac{p^2}{n(1-\epsilon)^4}\leq c$.
\end{Lemma}

With Lemma \ref{lem:estimation-huber}, the bound (\ref{eq:bias-control-est}) holds as long as $\frac{p^2}{n(1-\epsilon)^4}$ is sufficiently small. By Lemma \ref{lem:decor}, we immediately obtain the coverage guarantee of the confidence interval.
\begin{Theorem}\label{thm:main-cov}
Under the setting of Theorem \ref{thm:main}, for any constant $\alpha\in(0,1)$, the interval $[\wh{\beta}_L,\wh{\beta}_R]$ computed by Algorithm \ref{alg:CI} satisfies
$$\mathbb{P}\left(\beta\in [\wh{\beta}_L,\wh{\beta}_R]\right) \geq 1-\alpha,$$
as long as $\frac{p^2}{n(1-\epsilon)^4}\leq c$ for some $c>0$ depending on $\alpha$.
\end{Theorem}

\subsection{Length}\label{sec:length}

To bound the length of the confidence interval (\ref{eq:main-CI}), it suffices to show for any $\wt{\beta}\in\mathbb{R}$ such that $|\wt{\beta}-\beta|>r$, $\wt{\beta}$ does not belong to the set (\ref{eq:main-CI}) in the sense that
\begin{equation}
\left|\frac{G_n(\wt{\beta})}{\sqrt{\frac{1}{n}\sum_{i=1}^n\wt{x}_i^2}}\right|>1.5\sqrt{\frac{\log(2/\alpha)}{n}}.\label{eq:l-todo}
\end{equation}
This implies that the interval length is at most $2r$.
By triangle inequality,
\begin{equation}
|G_n(\wt{\beta})| \geq |-G_n(\wt{\beta})+G_n(\beta)|-|G_n(\beta)|.\label{eq:l-tria}
\end{equation}
It suffices to lower bound $|-G_n(\wt{\beta})+G_n(\beta)|$ and upper bound $|G_n(\beta)|$. The latter is directly implied by the coverage guarantee that $\beta$ belongs to (\ref{eq:main-CI}). That is,
\begin{equation}
|G_n(\beta)|\lesssim \frac{1}{\sqrt{n}}.\label{eq:l-tria-2}
\end{equation}

To lower bound $|-G_n(\wt{\beta})+G_n(\beta)|$, we note that the function $g(\cdot)$ satisfies the inequality
\begin{equation}
(g(t+\Delta)-g(t))\Delta \geq c \left(|\Delta|^2\wedge |\Delta|\right)\indi\{|t|\leq 1\}. \label{eq:curve-g}
\end{equation}
To see why (\ref{eq:curve-g}) is true, let us consider the two cases $|\Delta|\leq 1$ and $|\Delta|> 1$ separately. By symmetry, we could assume $\Delta\geq 0$ without loss of generality. When $|\Delta|\leq 1$, the function $g(\cdot)$ is locally linear so that $g(t+\Delta)-g(t)\gtrsim \Delta$. When $|\Delta|>1$, the monotonicity of $g(\cdot)$ implies $g(t+\Delta)-g(t)\gtrsim 1$.

Applying (\ref{eq:curve-g}) to the difference $|(-G_n(\wt{\beta})+G_n(\beta))(\beta-\wt{\beta})|$, we have
\begin{eqnarray}
\nonumber && |(-G_n(\wt{\beta})+G_n(\beta))(\beta-\wt{\beta})| \\
\nonumber &=& \left|\frac{1}{n}\sum_{i=1}^n\left(g\left((\gamma-\wh{\gamma})^Tw_i + z_i+\wt{x}_i(\beta-\wt{\beta})\right)-g\left((\gamma-\wh{\gamma})^Tw_i + z_i\right)\right)\wt{x}_i(\beta-\wt{\beta})\right| \\
\nonumber&\geq& c\frac{1}{n}\sum_{i=1}^n\left(|\wt{x}_i(\wt{\beta}-\beta)|^2\wedge |\wt{x}_i(\wt{\beta}-\beta)|\right)\indi\left\{|(\gamma-\wh{\gamma})^Tw_i+z_i|\leq 1\right\} \\
\label{eq:improved-later} &\geq& c\left(|\wt{\beta}-\beta|\wedge |\wt{\beta}-\beta|^2\right) \frac{1}{n}\sum_{i=1}^n(|\wt{x}_i|\wedge\wt{x}_i^2)\indi\left\{|(\gamma-\wh{\gamma})^Tw_i+z_i|\leq 1\right\}.
\end{eqnarray}
By Lemma \ref{lem:for-main}, we have
\begin{equation}
\frac{1}{n}\sum_{i=1}^n(|\wt{x}_i|\wedge\wt{x}_i^2)\indi\left\{|(\gamma-\wh{\gamma})^Tw_i+z_i|\leq 1\right\} \gtrsim 1-\epsilon, \label{eq:lower-noise-1-eps}
\end{equation}
with high probability. Intuitively, with the norm of $\gamma-\wh{\gamma}$ controlled by Lemma \ref{lem:estimation-huber}, the random variable $\indi\left\{|(\gamma-\wh{\gamma})^Tw_i+z_i|\leq 1\right\}$ has expectation at least of order $1-\epsilon$ by (\ref{eq:noise-sig}), and thus (\ref{eq:lower-noise-1-eps}) holds. Combining (\ref{eq:improved-later}) and (\ref{eq:lower-noise-1-eps}), we have
\begin{equation}
|-G_n(\wt{\beta})+G_n(\beta)| \gtrsim (1-\epsilon)\left(1\wedge |\wt{\beta}-\beta|\right).
\end{equation}
Together with (\ref{eq:l-tria}) and (\ref{eq:l-tria-2}), the inequality (\ref{eq:l-todo}) holds as long as $|\wt{\beta}-\beta|\gtrsim \frac{1}{\sqrt{n(1-\epsilon)^2}}$.

\begin{Theorem}\label{thm:main-length}
Under the setting of Theorem \ref{thm:main}, for any constant $\alpha\in(0,1)$, there exist some constants $c>0$ and $C>0$ depending on $\alpha$ such that the interval $[\wh{\beta}_L,\wh{\beta}_R]$ computed by Algorithm \ref{alg:CI} satisfies
$$\mathbb{P}\left(\wh{\beta}_R-\wh{\beta}_L\leq \frac{C}{\sqrt{n(1-\epsilon)^2}}\right) \geq 1-\alpha,$$
as long as $\frac{p^2}{n(1-\epsilon)^4}\leq c$.
\end{Theorem}

\subsection{Linear Functional}\label{sec:linearf}

Algorithm \ref{alg:CI} can also be applied to robust confidence interval construction of a general linear functional $u^Tb$ in the linear model (\ref{eq:mul-lin-comp}). The only modification is an additional rotation of the design matrix before applying Algorithm \ref{alg:CI}. Let $V=(v_1,\cdots,v_p)\in\mathbb{R}^{p\times p}$ be an orthonormal matrix with the first column given by $v_1=u/\|u\|$. Then, the regression model $X_i^Tb$ can be written as
$$X_i^Tb=X_i^T\left(VV^T\right)b=\frac{1}{\|u\|^2}X_i^Tu u^Tb+ \sum_{j=2}^pX_i^Tv_j v_j^Tb.$$
Then, the output of Algorithm \ref{alg:CI} with the same $y_i$ and the new covariates
$$\left(\frac{1}{\|u\|^2}X_i^Tu, X_i^Tv_2,\cdots,X_i^Tv_p\right)^T$$
is a valid robust confidence interval of $u^Tb$.

\section{Numerical Experiments}\label{sec:num}

\subsection{Numerical Setups}\label{sec:setupnum}

In this section, we evaluate the performance of Algorithm \ref{alg:CI}, together with several competitors, in the following synthetic data sets. The observations $\{(y_i,X_i)\}_{i=1}^n$ are generated according to the linear model $y_i=X_i^Tb+z_i$. The covariates are sampled from $X_i\overset{iid}\sim N(0,\Sigma)$ with $\Sigma_{jk}=0.6^{|j-k|}$. We consider two types of noise distributions:
\begin{enumerate}
\item \textbf{Gaussian}: $z_i\overset{iid}\sim (1-\epsilon)N(0,1)+\epsilon\left(\frac{1}{2}N(10^2,10^4)+\frac{1}{2}N(10^4,10^8)\right)$.
\item \textbf{Cauchy}: $z_i\overset{iid}\sim (1-\epsilon)C(0,1)+\epsilon\left(\frac{1}{2}N(10^2,10^4)+\frac{1}{2}N(10^4,10^8)\right)$.
\end{enumerate}
The Cauchy distribution $C(0,1)$ has density $\frac{1}{\pi(1+t^2)}$, and thus the second setting above involves both heavy tail and outliers. We set $p=20$ and vary $n\in\{200,1000\}$ and $\epsilon\in[0,0.8]$. Due to translation invariance of the linear model, we set $b=0$ without loss of generality.

The threshold level $1.5\sqrt{\log(2/\alpha)}$  used in Algorithm \ref{alg:CI} is to guarantee non-asymptotic coverage in Theorem \ref{thm:main}. One may wonder whether replacing $1.5\sqrt{\log(2/\alpha)}$  with the more aggressive Gaussian quantile $\Phi^{-1}(1-\alpha/2)$ works better in practice. While the Gaussian approximation works in the asymptotics of $\frac{p^2}{n(1-\epsilon)^4}\rightarrow 0$, our numerical results indicate that it should be used with caution when the sample size $n$ is small or $\epsilon$ is large. In addition to this variation of Algorithm \ref{alg:CI}, we also implement several methods in the literature. This includes the residual bootstrap (RB) of \cite{freedman1981bootstrapping}, the restructured regression (RR) of \cite{zhu2018linear}, the cyclic permutation test (CPT) of \cite{lei2021assumption}, the convex hull method (HulC) of \cite{kuchibhotla2024hulc} and the residual permutation test (RPT) of \cite{wen2025residual}. We also report the performance of the classical confidence interval computed from the t-statistic centering at the ordinary least-squares estimator (OLS). Among the above methods, RR \citep{zhu2018linear}, CPT \citep{lei2021assumption} and RPT \citep{wen2025residual} are originally proposed to test the null hypothesis that a coefficient is zero, which also lead to confidence intervals by inverting the tests.

\subsection{Implementation Details}

The residual bootstrap \citep{freedman1981bootstrapping} is implemented with Huber regression
\begin{equation}
\min_b\frac{1}{n}\sum_{i=1}^n\rho(y_i-X_i^Tb),\label{eq:ghr}
\end{equation}
where $\rho(\cdot)$ is the same Huber loss used in Algorithm \ref{alg:CI}. The bootstrap sample size is taken as $200$.
The restructured regression method has two versions in \cite{zhu2018linear}, and we implement the one with the knowledge of the design covariance.
The cyclic permutation test \citep{lei2021assumption} is implemented with the recommended preordering computed using a genetic algorithm. The genetic algorithm is
computed with $1000$ random samples. The HulC method \citep{kuchibhotla2024hulc} is also implemented with the Huber regression (\ref{eq:ghr}), together with a recommended randomized choice of sample splitting to prevent the interval from being overly conservative. The residual permutation test \citep{wen2025residual} is implemented with the more powerful version using the empirical p-values, and the tuning parameters are set as recommended in the original work.

\subsection{Results and Discussion} \label{sec:res and dis}

We report the coverage and length properties of confidence interval construction for the first regression coefficient with level $1-\alpha=0.95$. Figure \ref{fig:coverage} reports the coverage probabilities, and Figure \ref{fig:length} reports the average lengths of the intervals with coverage. Each point in the plots is computed from $500$ independent experiments. Besides RB, RR, CPT, HulC, RPT and OLS, our methods are denoted as Alg1 for the original version of Algorithm \ref{alg:CI} and Alg1-GA for the variant of Algorithm \ref{alg:CI} with $1.5\sqrt{\log(2/\alpha)}$  replaced by $\Phi^{-1}(1-\alpha/2)$ using Gaussian approximation. Since CPT requires $n/p\geq \alpha^{-1}-1$ \citep{lei2021assumption}, it is not included in the setting with $n=200$ and $p=20$.

\begin{figure}[ht]
  \centering
  \includegraphics[width=\linewidth]{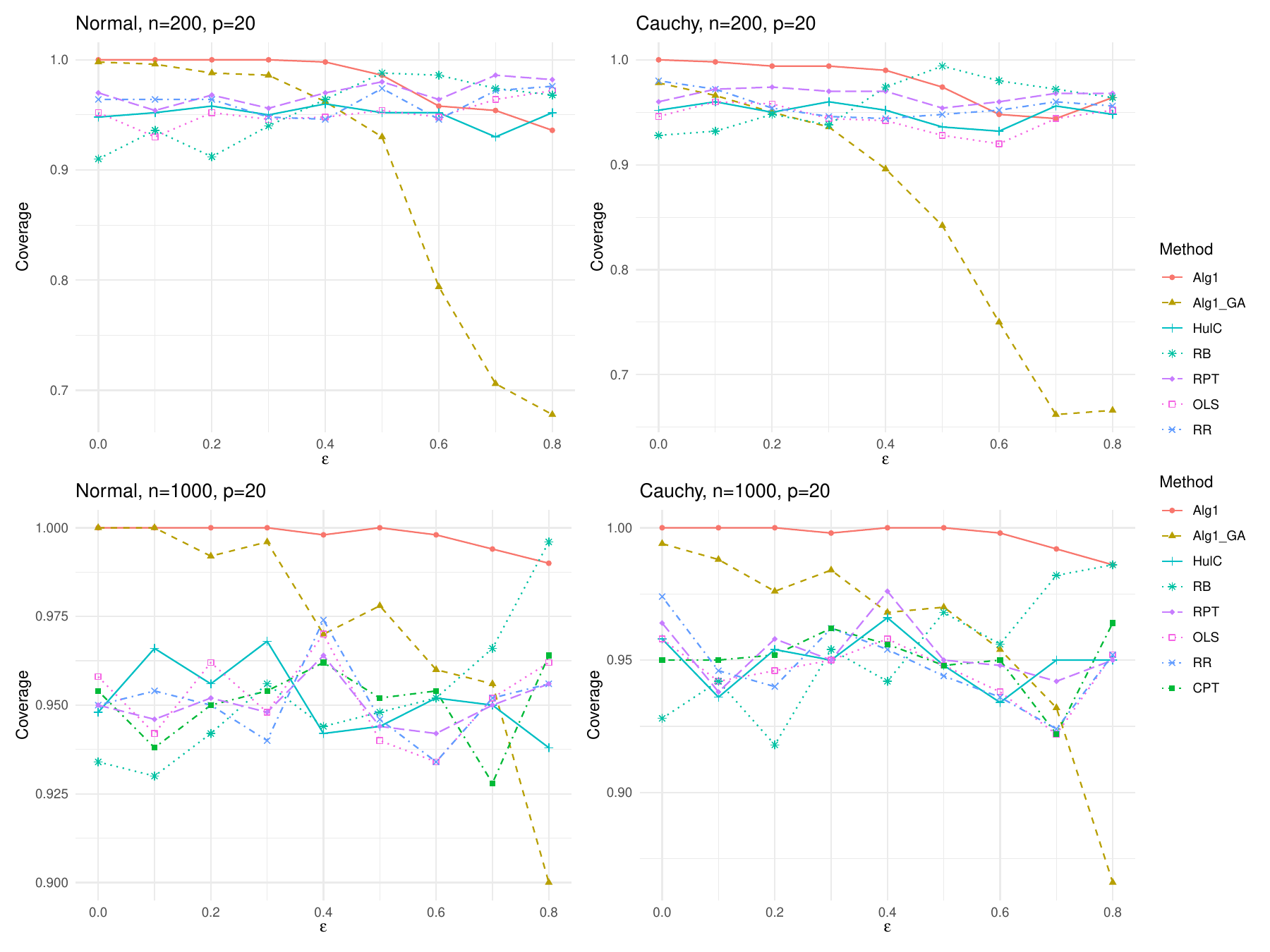}
  \caption{Coverage against $\epsilon$ across settings.}
  \label{fig:coverage}
\end{figure}

\begin{figure}[ht]
  \centering
  \includegraphics[width=\linewidth]{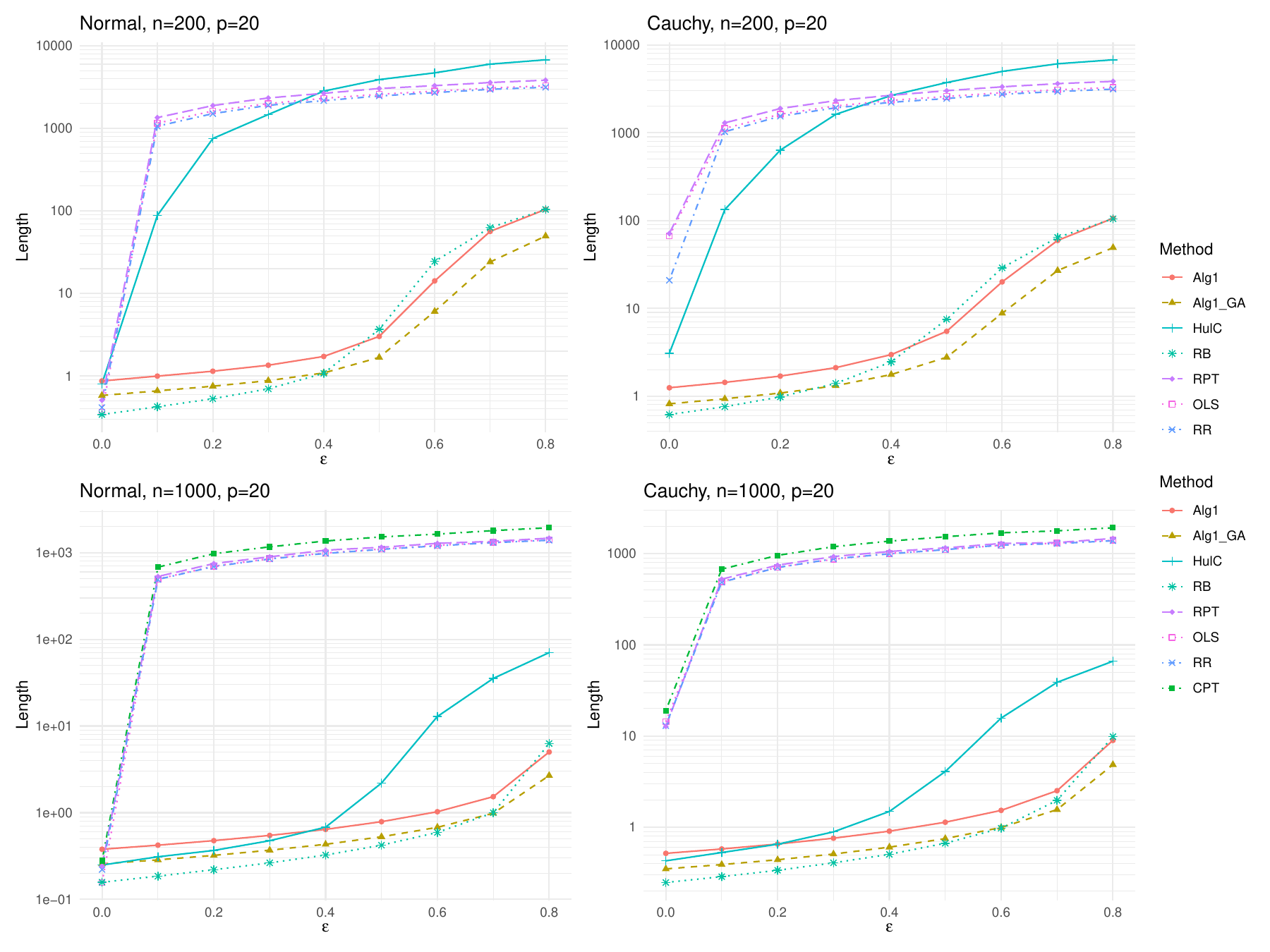}
  \caption{Average interval length against $\epsilon$ across settings.}
  \label{fig:length}
\end{figure}

Figure \ref{fig:coverage} shows that most methods achieve coverage at levels around $0.95$. In particular, since the proposed Algorithm \ref{alg:CI} (Alg1) is designed to have non-asymptotic coverage guarantee, its coverage probability as a function of $\epsilon$ consistently stays above the nominal level, while other methods all have fluctuations. In comparison, Alg1-GA, which is the variant of Algorithm \ref{alg:CI} with $1.5\sqrt{\log(2/\alpha)}$  replaced by $\Phi^{-1}(1-\alpha/2)$, has coverage below the nominal level for large values of $\epsilon$. This indicates that Gaussian approximation does not hold when the data has a large proportion of outliers, and the threshold level $1.5\sqrt{\log(2/\alpha)}$ based on non-asymptotic concentration is necessary.

Figure \ref{fig:length} shows that the length properties are drastically different across different methods. The performances of restructured regression (RR), cyclic permutation test (CPT) and residual permutation test (RPT) are similar to that of OLS. These methods do not seem to be robust, given the interval lengths reported in settings with large $\epsilon$ or Cauchy noise. Though CPT and RPT are proved to achieve coverage properties even when $\epsilon=1$ \citep{lei2021assumption,wen2025residual}, the length guarantees require a moment condition \citep{wen2025residual} and our experiments show that they do not work with arbitrary outliers. The convex hull method (HulC) performs reasonably well in the setting of $n=1000$ and $p=20$ for $\epsilon$'s that are not too large. However, when $n=200$, it is not competitive due to the necessary sample splitting step in the interval construction. The classical residual bootstrap (RB) works surprisingly well across all settings. It is both efficient when $\epsilon$ is small and robust when $\epsilon$ is large, even though there is not theoretical guarantee for residual bootstrap and we do not think it is a provable method for arbitrary contamination distribution. Compared with Algorithm \ref{alg:CI}, the residual bootstrap tends to undercover even when $\epsilon$ is very small. The most severe drawback of residual bootstrap is its computational cost, which is quite typical for a method based on resampling. Its computational time is around 25--180 times that of Algorithm \ref{alg:CI} depending on the outlier configuration.

In summary, we conclude that Algorithm \ref{alg:CI} behaves consistently well in all settings in terms of both coverage and length properties, which agrees with the theoretical guarantee established in Theorem \ref{thm:main}. Compared with residual bootstrap and other resampling algorithms, Algorithm \ref{alg:CI} has far better computational cost, which makes it more scalable for large data sets.

\subsection{Sensitivity to Model Parameters}\label{sec:num-add}

We conduct additional simulations to assess the sensitivity of the proposed method to key parameters, including sample size, scaling parameter, and the covariate distribution. The observations are generated according to the same setup described in Section \ref{sec:setupnum} unless otherwise mentioned. The noise distribution is Gaussian and the confidence level is set as $1-\alpha=0.95$. All experiments are implemented by Alg1 (the original version of Algorithm \ref{alg:CI}). Since we observe that the coverage property is always satisfied, we only report the length.

\begin{figure}[htbp]
  \centering
  \includegraphics[width=\linewidth] {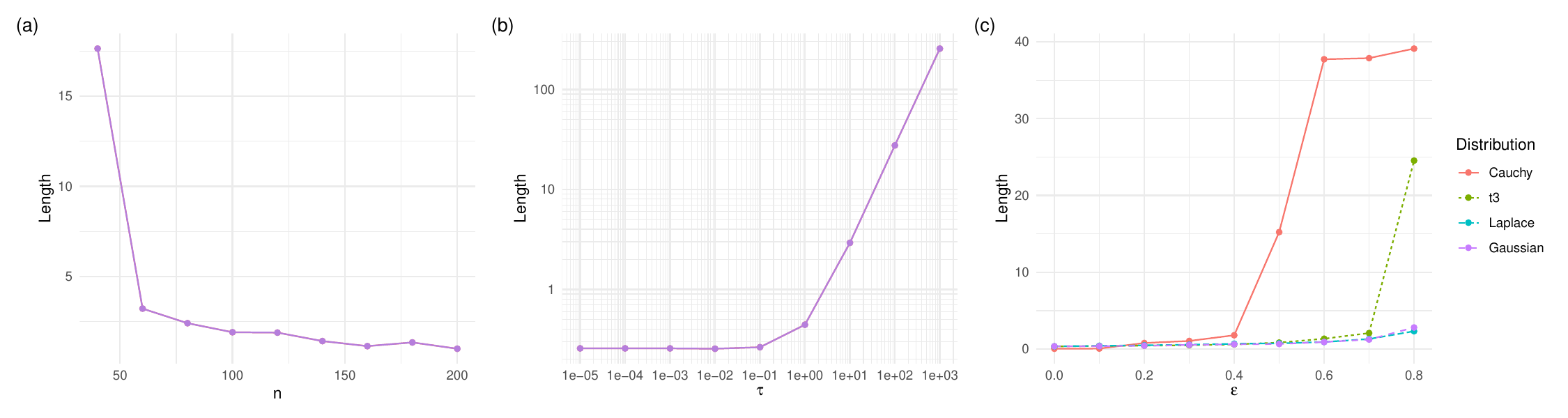}
  \caption{Average interval length against model parameters and covariate distribution.}
  \label{fig:len_plus}
\end{figure}

\prettyref{fig:len_plus}(a) varies $n\in[40,200]$ with $p=20$ and $\epsilon=0.2$ being fixed. The goal is to check how soon the robustness property of Algorithm \ref{alg:CI} kicks in as the sample size increases. For this setting with moderate dimension, $n\geq 50$ would be sufficient.

We use \prettyref{fig:len_plus}(b) to illustrate the effect of the algorithm scaling. Suppose we do not know that $\sigma^2=1$, Algorithm \ref{alg:CI} should be applied with the scaling $g(\cdot)\rightarrow g(\cdot/\tau)$ and $\rho(\cdot)\rightarrow \rho(\cdot/\tau)$ for some $\tau>0$. Then, Theorem \ref{thm:main} continues to hold as long as $\tau\lesssim 1$ and $\frac{p^2}{n(1-\epsilon)^4\tau^4}$ is sufficiently small. \prettyref{fig:len_plus}(b) shows that this is indeed the case. With $n=1000$, $p=20$ and $\epsilon=0.2$ being fixed, the length of the confidence interval stabilizes as long as $\tau$ is set to be sufficiently small. In fact, there does not seem to exist a regime where $\tau$ is too small, despite that we still need the condition $\frac{p^2}{n(1-\epsilon)^4\tau^4}<c$ theoretically.

\prettyref{fig:len_plus}(c) investigates the effect of the covariate distribution. In particular, we generate $X_i$ so that $\Sigma^{-1/2}X_i$ consists of $p$ independent standard Cauchy, $t$-distribution (3 degrees of freedom), Laplace and Gaussian with $\Sigma_{jk}=0.6^{|j-k|}$. The simulation is conducted with $p=20$ and $n=1000$ with $\epsilon$ varying from $0$ to $0.8$. It is clear from \prettyref{fig:len_plus}(c) that a lighter tail tends to be more robust, which is naturally expected. On the other hand, the result is not sensitive to the distribution tail. The breakdown point of the case with $t$-distribution still exceeds $\epsilon=0.7$.

\begin{figure}[htbp]
  \centering
  \includegraphics[width=\linewidth,height=0.33\linewidth] {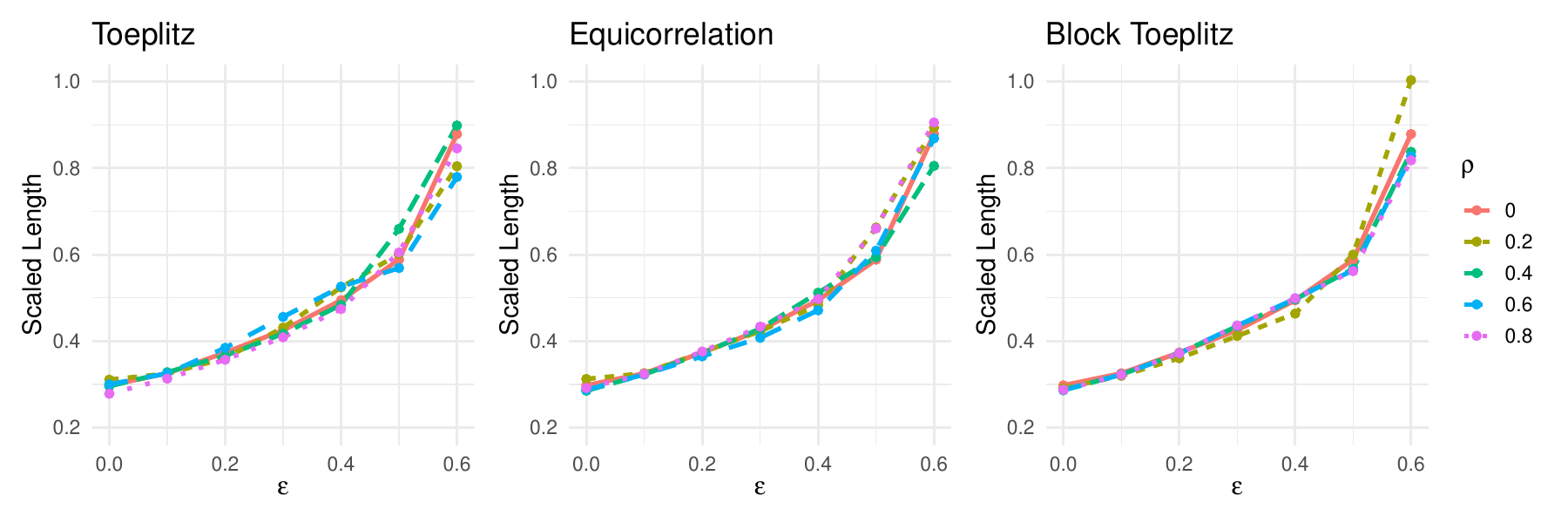}
  \caption{Scaled average interval length against $\epsilon$ with various correlation structures.}
  \label{fig:len vs cov}
\end{figure}

We finally investigate the effect of the correlation structure of the covariates. Let $X_i\overset{iid}\sim N(0,\Sigma)$. We consider three types of covariance matrices: equicorrelation, $\Sigma_{jk}=\rho+(1-\rho)\indi\{j=k\}$; Toeplitz, $\Sigma_{jk}=\rho^{|j-k|}$; and block Toeplitz, where $\rho^{|j-k|}$ within each of five equal-sized independent blocks. For a given $\Sigma$, the effective noise level for inference on the first regression coefficient is $\sqrt{(\Sigma^{-1})_{11}}$. Accordingly, we report the length of the confidence interval scaled by the inverse of this noise level. The simulation results for $p=20$ and $n=1000$ are shown in Figure \ref{fig:len vs cov}. The curves of the scaled confidence interval length as functions of $\epsilon$ are well aligned across different values of $\rho$ and across the three covariance structures.

\subsection{A Real Data Example}\label{sec:real}

We apply our robust algorithm to the Boston Housing data. The data were originally collected for the study of \cite{Harrison_1978} on housing prices. The variable description and linear model can be found at Table IV and the appendix of \cite{Harrison_1978}. The variables are transformed and standardized before regression.

For each regression coefficient, instead of computing its confidence interval, we modify Algorithm \ref{alg:CI} to compute a point estimator and a $p$-value. With $(\wt{x_i},\wt{y}_i)$ computed by the first four steps of Algorithm \ref{alg:CI}, the point estimator is found as the solution to $\sum_{i=1}^ng(\wt{y}_i-\wt{x}_i\beta)\wt{x}_i=0$, and the $p$-value is given by the formula
$$2\exp\left(-\frac{\left|\sum_{i=1}^ng(\wt{y}_i)\wt{x}_i\right|^2}{2.25\sum_{i=1}^n\wt{x}_i^2}\right).$$
We first apply OLS and Algorithm \ref{alg:CI} to the original data. Then, Algorithm \ref{alg:CI} is applied again to the contaminated data after adding $\frac{1}{2}N(10^2,10^4)+\frac{1}{2}N(10^4,10^8)$ to each response variable with probability $0.1$. The estimated coefficients and $p$-values are reported in \prettyref{tab:coef and p for Boston}. 
\begin{table}[!h]
\centering
\caption{Estimated coefficients and $p$-values for the Boston Housing data.}
\label{tab:coef and p for Boston}
\centering
\fontsize{10}{12}\selectfont
\begin{tabular}[t]{lcccccc}
\toprule
\multicolumn{1}{c}{ } & \multicolumn{2}{c}{OLS} & \multicolumn{2}{c}{Alg1} & \multicolumn{2}{c}{Alg1\_c} \\
\cmidrule(l{3pt}r{3pt}){2-3} \cmidrule(l{3pt}r{3pt}){4-5} \cmidrule(l{3pt}r{3pt}){6-7}
Covariate & Coef. & $p$-value & Coef. & $p$-value & Coef. & $p$-value\\
\midrule
age & 0.006 & 0.863 & -0.039 & 1 & -0.039 & 1\\
chas & 0.057 & 0.006 & 0.006 & 1 & 0.001 & 1\\
crim & -0.250 & <0.001 & -0.268 & <0.001 & -0.318 & 0.023\\
dis & -0.252 & <0.001 & -0.283 & 0.002 & -0.228 & 0.137\\
indus & 0.004 & 0.919 & -0.053 & 0.805 & 0.120 & 0.829\\
nox & -0.217 & <0.001 & -0.169 & 0.007 & -0.143 & 0.306\\
ptratio & -0.165 & <0.001 & -0.137 & <0.001 & -0.182 & <0.001\\
rad & 0.205 & <0.001 & 0.135 & 0.013 & 0.234 & 0.037\\
rm & 0.141 & <0.001 & 0.203 & 0.001 & 0.281 & 0.043\\
tax & -0.173 & <0.001 & -0.138 & 0.228 & -0.300 & 0.069\\
zn & 0.005 & 0.874 & -0.004 & 1 & 0.042 & 1\\
\bottomrule
\end{tabular}
\end{table}

According to the table, the estimated coefficients are relatively stable in terms of magnitude and sign. The inference, however, shows more variability. The OLS identifies many covariates as significant,
while Algorithm \ref{alg:CI} is more conservative and identifies fewer significant covariates. Since the residual Q--Q plot is quite heavy-tailed, the results of OLS are less reliable. Moreover, after the data is contaminated, Algorithm \ref{alg:CI} selects even fewer covariates as significant. The covariates \textsf{crim}, \textsf{ptratio}, \textsf{rad} and \textsf{rm} are significant under all scenarios, revealing a strong dependence of housing price on crime rate, pupil-teacher ratio, accessibility to radial highways, and average room number.

\section{Beyond Oblivious Contamination}\label{sec:beyond}

The robust regression setting (\ref{eq:mul-lin-comp}) with noise variables (\ref{eq:noise-sig}) has two crucial conditions: first, the contamination does not affect the covariates $X_i$; second, the contamination is independent of $X_i$. While relaxing these two conditions is certainly possible for estimation, we show in this section that the two conditions are essentially necessary for adaptive inference.

\textbf{Joint Contamination on $(X_i,y_i)$} Let us consider the univariate setting where samples $\{(x_i,y_i)\}_{i=1}^n$ are i.i.d. generated from the Huber contamination model $(1-\epsilon)P_{\beta,\tau,\sigma}+\epsilon Q$. Here, $P_{\beta,\tau,\sigma}$ stands for the joint distribution of $(x_i,y_i)$ whose sampling process is $x_i\sim N(0,\tau^2)$ and $y_i|x_i\sim N(\beta x_i,\sigma^2)$. In this setting where both $x_i$ and $y_i$ are contaminated, a shrinking adaptive confidence interval of $\beta$ is impossible.
\begin{Proposition}\label{prop:imp-joint}
There exists some $Q$ and some sufficiently small constant $c$, such that $P_{0,1,1}=\frac{2}{3}P_{c,0.99,0.99}+\frac{1}{3}Q$.
\end{Proposition}

The proposition shows that a linear model with $\beta=0$ and no contamination can be represented as a linear model with $\beta=c$ under $1/3$ contamination. Here, the constant $c$ can be taken as $\frac{1-0.99^2}{2\times 0.99}$. Therefore, any confidence interval that is agnostic between $\epsilon=0$ and $\epsilon=1/3$ must contain both $0$ and $c$, and hence has a constant length independent of the sample size.

\textbf{Contamination Depending on $X_i$} Even if the contamination only affects the response variables $y_i$, it will still be impossible to construct a shrinking adaptive confidence interval unless the contamination is independent of $X_i$. Let us consider the univariate setting, and the relation of the pair $(x_i,y_i)$ follows
\begin{equation}
y_i=\beta x_i+z_i+o_i,\label{eq:additive-outlier-dep}
\end{equation}
where $\beta\in\mathbb{R}$ and $o_i$ is an outlier variable that may depend on $x_i$. The model (\ref{eq:additive-outlier-dep}) has been widely studied in the literature on robust regression \cite{gannaz2007robust,mccann2007robust,she2011outlier}. When $o_i$ is independent of $x_i$, the sum $z_i+o_i$ follows (\ref{eq:noise-sig}) and is thus a special case of the robust regression with oblivious contamination.

We introduce the $P_{\beta,\sigma,Q(\cdot)}$ as the joint distribution of $(x_i,y_i)$, whose sampling process is given by the linear equation (\ref{eq:additive-outlier-dep}), where $x_i\sim N(0,1)$ and then $z_i\sim N(0,\sigma^2)$ independently of $x_i$. Finally, $o_i|x_i,z_i\sim Q(x_i,z_i)$ for the conditional distribution $Q(\cdot)$ given $x_i$ and $z_i$. The contamination proportion is  $\epsilon=\mathbb{E}Q(o_i\neq 0|x_i,z_i)$.

\begin{Proposition}\label{prop:imp-depend}
There exist $Q_0(\cdot)$, $Q_1(\cdot)$ and some sufficiently small constant $c$, such that $\mathbb{E}Q_0(o_i\neq 0|x_i,z_i)=0$, $\mathbb{E}Q_1(o_i\neq 0|x_i,z_i)=0.4$, and $P_{0,1,Q_0(\cdot)}=P_{c,0.99,Q_1(\cdot)}$.
\end{Proposition}

Proposition \ref{prop:imp-depend} admits an interpretation similar to that of Proposition \ref{prop:imp-joint}. In particular, a linear model with $\beta=0$ and no contamination can be represented as a linear model with $\beta=c$ under $40\%$ contamination where $c$ can be taken as 0.07. Consequently, any confidence interval that is agnostic about the contamination proportion must have a constant length independent of the sample size.

\section{Proofs}\label{sec:proof}

This section collects all technical proofs. We will first derive error bounds of $\wh{\alpha}$ and $\wh{\gamma}$ in Section \ref{sec:pf-alpha-gamma}. The main results (Theorem \ref{thm:main-cov}, Theorem \ref{thm:main-length} and Theorem \ref{thm:main}) will be proved in Section \ref{sec:pf-main}. We will then prove Proposition \ref{prop:loc-med} and Proposition \ref{prop:cov-gaus} in Section \ref{sec:pf-prop}, and prove Proposition \ref{prop:imp-joint} and Proposition \ref{prop:imp-depend} in Section \ref{sec:pf-by}. Finally, additional technical lemmas will be proved in Section \ref{sec:pf-lem}.

\subsection[Bounding Errors of alpha-hat and gamma-hat]%
{Bounding Errors of $\widehat{\alpha}$ and $\widehat{\gamma}$}
\label{sec:pf-alpha-gamma}

Writing the joint covariance of $X_i=(x_i,w_i^T)^T$ as
$$\Sigma=\begin{pmatrix}
\Sigma_{xx} & \Sigma_{xw} \\
\Sigma_{wx} & \Sigma_{ww}
\end{pmatrix}.$$
The least-squares estimator $\wh{\alpha}$ can be regarded as the sample version of $\alpha=\Sigma_{ww}^{-1}\Sigma_{wx}$.

\begin{Lemma} \label{lm:error in first-stage}
Under the setting of Theorem \ref{thm:main}, for any constant $\delta\in(0,1)$, there exist some constants $c>0$ and $C>0$ depending on $\delta$ such that
$$\|\wh{\alpha}-\alpha\|\leq C\sqrt{\frac{p}{n}},$$
with probability at least $1-\delta$  as long as $\frac{p^2}{n(1-\epsilon)^4}\leq c$.
\end{Lemma}
\begin{proof}
\prettyref{as:bounded moments} implies $\opnorm{\Sigma}\leq 2K_1^2$ and $\opnorm{\Sigma^{-1}}\leq K_3$.
Define
$$\wh{\Sigma}=\frac{1}{n}\sum_{i=1}^nX_iX_i^T=\begin{pmatrix}
\wh{\Sigma}_{xx} & \wh{\Sigma}_{xw} \\
\wh{\Sigma}_{wx} & \wh{\Sigma}_{ww}
\end{pmatrix}.$$
Given the sub-Gaussianity in \prettyref{as:bounded moments}, we have $\opnorm{\wh{\Sigma}-\Sigma}\leq C_1\sqrt{\frac{p}{n}}$ with high probability (see Lemma \ref{lm:fourth moment of covariate} stated in Section \ref{sec:pf-lem}). Since $\wh{\alpha}=\wh{\Sigma}_{ww}^{-1}\wh{\Sigma}_{wx}$, we have
$$\|\wh{\alpha}-\alpha\|\leq \opnorm{\wh{\Sigma}_{ww}^{-1}}\|\wh{\Sigma}_{wx}-\Sigma_{wx}\|+\opnorm{\wh{\Sigma}_{ww}^{-1}-\Sigma_{ww}^{-1}}\|\Sigma_{wx}\|,$$
where $\|\wh{\Sigma}_{wx}-\Sigma_{wx}\|\leq \opnorm{\wh{\Sigma}-\Sigma}\leq C_1\sqrt{\frac{p}{n}}$ and $\|\Sigma_{wx}\|\leq \opnorm{\Sigma}\leq 2K_1^2$. Since $\opnorm{\wh{\Sigma}_{ww}-\Sigma_{ww}}\leq \opnorm{\wh{\Sigma}-\Sigma}\leq C_1\sqrt{\frac{p}{n}}$, we know that both $\opnorm{\wh{\Sigma}_{ww}}$ and $\opnorm{\wh{\Sigma}_{ww}^{-1}}$ are bounded by constants. Therefore, $\opnorm{\wh{\Sigma}_{ww}^{-1}-\Sigma_{ww}^{-1}}\leq \opnorm{\wh{\Sigma}_{ww}^{-1}}\opnorm{\wh{\Sigma}_{ww}-\Sigma_{ww}}\opnorm{\Sigma_{ww}^{-1}}\leq C_2\sqrt{\frac{p}{n}}$. Combining these bounds, we obtain the desired rate for $\|\wh{\alpha}-\alpha\|$.
\end{proof}

Next, we will prove Lemma \ref{lem:estimation-huber} that bounds the error of $\wh{\gamma}$. This requires the following lemma on the random vector $\wt{X}_i=(\wt{x}_i,w_i^T)^T\in\mathbb{R}^p$.

\begin{Lemma}\label{lem:for-gamma}
Under the setting of Theorem \ref{thm:main}, for any constant $\delta\in(0,1)$, there exist some constants $C,c>0$ depending on $\delta$ such that 
\begin{enumerate}
\item $\left\|\frac{1}{n}\sum_{i=1}^n\rho'(z_i)\wt{X}_i\right\|\leq C\sqrt{\frac{p}{n}}$,
\item $\inf_{\|v\|=1}\frac{1}{n}\sum_{i=1}^n(|v^T\wt{X}_i|^2\wedge |v^T\wt{X}_i|)\indi\{|z_i|\leq 1\}\geq c(1-\epsilon),$
\end{enumerate}
with probability at least $1-\delta$.
\end{Lemma}
The proof of Lemma \ref{lem:for-gamma} will be given in Section \ref{sec:pf-lem}. Now we are ready to prove Lemma \ref{lem:estimation-huber}.

\begin{proof}[Proof of Lemma \ref{lem:estimation-huber}]
Recall that $\gamma=\beta\wh{\alpha}+\theta$, and we can write $y_i=\beta\wt{x}_i+\gamma^Tw_i+z_i$. Define $\eta=(\beta,\gamma^T)^T\in\mathbb{R}^p$, and the model becomes
$$y_i=\wt{X}_i^T\eta+z_i.$$
Similarly, define $\wh{\eta}=(\wh{\beta},\wh{\gamma}^T)^T\in\mathbb{R}^p$, where $\wh{\beta}$ and $\wh{\gamma}$ solve the optimization $\min_{\beta,\gamma}\frac{1}{n}\sum_{i=1}^n\rho(y_i-\beta \wt{x}_i-\gamma^Tw_i)$. Since $\|\wh{\gamma}-\gamma\|\leq \|\wh{\eta}-\eta\|$, it suffices to bound $\|\wh{\eta}-\eta\|$. The definition of $\wh{\eta}$ implies 
\begin{eqnarray}
        \nonumber 0&=&\frac{1}{n}\sum_{i=1}^n \rho'(y_i-\wt{X}_i^T\wh{\eta} )\wt{X}_i \\
        \nonumber  &=& \frac{1}{n}\sum_{i=1}^n \rho'(\wt{X}_i^T(\eta-\wh{\eta})+z_i)\wt{X}_i  \\
        \label{eq:score huber}  &=& \frac{1}{n}\sum_{i=1}^n \rho'(z_i)\wt{X}_i+ \frac{1}{n}\sum_{i=1}^n (\rho'(\wt{X}_i^T(\eta-\wh{\eta})+z_i)-\rho'(z_i))\wt{X}_i.
\end{eqnarray}
It is straightforward to check that $\rho'(\cdot)$ also satisfies the inequality \eqref{eq:curve-g}. That is, 
\begin{equation*}
(\rho'(t+\Delta)-\rho'(t))\Delta \geq c \left(|\Delta|^2\wedge |\Delta|\right)\indi\{|t|\leq 1\},
\end{equation*}
for some constant $c>0$. Together with Lemma \ref{lem:for-gamma}, we have 
\begin{eqnarray}
    \nonumber &&\frac{1}{n}\sum_{i=1}^n (\rho'(\wt{X}_i^T(\eta-\wh{\eta})+z_i)-\rho'(z_i))\wt{X}_i^T(\eta-\wh{\eta}) \\
    \nonumber&\ge& c \frac{1}{n}\sum_{i=1}^n  \left(|\wt{X}_i^T(\eta-\wh{\eta})|^2\wedge |\wt{X}_i^T(\eta-\wh{\eta})|\right)\indi\{|z_i|\leq 1\} \\
    \nonumber &\ge& c  \|\eta-\wh{\eta}\|^2\land \|\eta-\wh{\eta}\|  \inf_{\|v\|=1}\frac{1}{n}\sum_{i=1}^n(|v^T\wt{X}_i|^2\wedge |v^T\wt{X}_i|)\indi\{|z_i|\leq 1\} \\
    \label{ineq:lower huber} &\ge& c_2(1-\epsilon) \left(\|\eta-\wh{\eta}\|^2 \land \|\eta-\wh{\eta}\|\right).
\end{eqnarray}
On the other hand, 
\begin{equation}
\left|\frac{1}{n}\sum_{i=1}^n \rho'(z_i)\wt{X}_i (\eta-\wh{\eta})\right|\leq \|\eta-\wh{\eta}\|  \left\|\frac{1}{n}\sum_{i=1}^n\rho'(z_i)\wt{X}_i\right\| \leq C\sqrt{\frac{p}{n}} \|\eta-\wh{\eta}\|, \label{ineq:upper huber}
\end{equation}
where the last inequality is again from Lemma \ref{lem:for-gamma}. Combining \eqref{eq:score huber}, \eqref{ineq:lower huber} and \eqref{ineq:upper huber}, we have
$c_2(1-\epsilon) \left(\|\eta-\wh{\eta}\| \land 1\right)\le  C\sqrt{\frac{p}{n}}$, which leads to the desired conclusion when $\frac{p}{n(1-\epsilon)^2}$ is small.
\end{proof}

\subsection{Proof of Main Results} \label{sec:pf-main}

We first state the following two lemmas, whose proofs will be given in Section \ref{sec:pf-lem}.
\begin{Lemma}\label{lem:for-main}
Under the setting of Theorem \ref{thm:main}, for any constant $\delta\in(0,1)$, there exist some constants $C,c,c'>0$ depending on $\delta$ such that
\begin{enumerate}
\item $\left\|\frac{1}{n}\sum_{i=1}^n\left(g'\left(z_i\right)-\mathbb{E}g'\left(z_i\right)\right)\wt{x}_iw_i\right\|\leq C\sqrt{\frac{p}{n}}$,
\item $\opnorm{\frac{1}{n}\sum_{i=1}^n|\wt{x}_i|w_iw_i^T}\leq C$,
\item $\frac{1}{n}\sum_{i=1}^n(|\wt{x}_i|\wedge\wt{x}_i^2)\indi\left\{|(\gamma-\wh{\gamma})^Tw_i+z_i|\leq 1\right\} \geq c'(1-\epsilon)$,
\item $c'\leq \frac{1}{n}\sum_{i=1}^n\wt{x}_i^2\leq C$,
\end{enumerate}
with probability at least $1-\delta$  as long as $\frac{p^2}{n(1-\epsilon)^4}\leq c$.
\end{Lemma}

\begin{Lemma}\label{lem:self}
Under the setting of Theorem \ref{thm:main}, for any constant $\alpha\in(0,1)$,
$$\mathbb{P}\left(\left|\frac{\frac{1}{n}\sum_{i=1}^ng(z_i)\wt{x}_i}{\sqrt{\frac{1}{n}\sum_{i=1}^n\wt{x}_i^2}}\right|>1.45\sqrt{\frac{\log(2/\alpha)}{n}}\right)\leq 0.97\alpha,$$
as long as $\frac{p^2}{n(1-\epsilon)^4}\leq c$ for some constant $c>0$ depending on $\alpha$.
\end{Lemma}

\begin{proof}[Proof of Theorem \ref{thm:main-cov}]
By Lemma \ref{lem:decor}, it suffices to bound (\ref{eq:cov-general}) and (\ref{eq:bias-control-est}), which are implied by Lemma \ref{lem:self} and Lemma \ref{lem:estimation-huber}, respectively. Combining the bounds and choosing the $\delta$ in Lemma \ref{lem:decor} to be sufficiently small, we have $\frac{|G_n(\beta)|}{\sqrt{\frac{1}{n}\sum_{i=1}^n\wt{x}_i^2}}\leq 1.5\sqrt{\frac{\log(2/\alpha)}{n}}$ with probability at least $1-\alpha$, which implies coverage.
\end{proof}

\begin{proof}[Proof of Theorem \ref{thm:main-length}]
Following the arguments in Section \ref{sec:length}, it suffices to show (\ref{eq:l-todo}) holds uniformly over all $\wh{\beta}$ such that $|\wt{\beta}-\beta|>r$. Since the function $G_n(\cdot)$ is monotone, we only need to show (\ref{eq:l-todo}) for $\wt{\beta}\in\{\beta-r,\beta+r\}$. Since $|G_n(\wt{\beta})| \geq |-G_n(\wt{\beta})+G_n(\beta)|-|G_n(\beta)|$, we need to lower bound $|-G_n(\wt{\beta})+G_n(\beta)|$ and upper bound $|G_n(\beta)|$. Applying Theorem \ref{thm:main-cov} with its $\alpha$ replaced by some small $\delta$ to be determined later, we have $|G_n(\beta)|\leq \frac{C_1}{\sqrt{n}}\sqrt{\frac{1}{n}\sum_{i=1}^n\wt{x}_i^2}$ with probability at least $1-\delta$. Next, we use (\ref{eq:improved-later}) followed by Lemma \ref{lem:for-main}, and obtain $|-G_n(\wt{\beta})+G_n(\beta)| \geq c_1(1-\epsilon)\left(1\wedge |\wt{\beta}-\beta|\right)$. Therefore, for $\wt{\beta}\in\{\beta-r,\beta+r\}$, we have $|G_n(\wt{\beta})|\geq c_1(1-\epsilon)(1\wedge r)-\frac{C_1}{\sqrt{n}}\sqrt{\frac{1}{n}\sum_{i=1}^n\wt{x}_i^2}$. In order that (\ref{eq:l-todo}) holds, we need $c_1(1-\epsilon)(1\wedge r)\geq \frac{C_1+1.5\sqrt{\log(2/\alpha)}}{\sqrt{n}}\sqrt{\frac{1}{n}\sum_{i=1}^n\wt{x}_i^2}$. Since $\frac{1}{n}\sum_{i=1}^n\wt{x}_i^2$ is bounded by a constant according to Lemma \ref{lem:for-main}, it is sufficient to set $r=\frac{C_2}{\sqrt{n(1-\epsilon)^2}}$ for some sufficiently large constant $C_2>0$. In summary, (\ref{eq:l-todo}) holds for $\wt{\beta}\in\{\beta-r,\beta+r\}$ with probability at least $1-C'\delta$. The proof is complete by setting $\delta=\alpha/C'$.
\end{proof}

\begin{proof}[Proof of Theorem \ref{thm:main}]
With slight modifications of the proofs of Theorem \ref{thm:main-cov} and Theorem \ref{thm:main-length}, it can be shown that $\beta\in[\wh{\beta}_L,\wh{\beta}_R]$ holds with probability at least $1-0.99\alpha$ and $\wh{\beta}_R-\wh{\beta}_L\leq \frac{C}{\sqrt{n(1-\epsilon)^2}}$ holds with probability at least $1-0.01\alpha$. Therefore, by the union bound, the two events hold simultaneously with probability at least $1-\alpha$.
\end{proof}

\subsection{Proofs of Proposition \ref{prop:loc-med} and Proposition \ref{prop:cov-gaus}} \label{sec:pf-prop}

\begin{proof}[Proof of Proposition \ref{prop:loc-med}]
The coverage property is equivalent to (\ref{eq:coverage-event}), which is implied by Hoeffding's inequality (Lemma \ref{lm:Hoeffding} stated in Section \ref{sec:pf-lem}). To derive the length guarantee, it suffices to show that for any $|\wt{\beta}-\beta|>\frac{C}{2\sqrt{n(1-\epsilon)^2}}$,
$$
\frac{1}{n}\sum_{i=1}^n\indi\{y_i\leq \wt{\beta}\} < \frac{1}{2}- \sqrt{\frac{\log(2/\alpha)}{2n}}\quad\text{ or }\quad\frac{1}{n}\sum_{i=1}^n\indi\{y_i< \wt{\beta}\} > \frac{1}{2}+ \sqrt{\frac{\log(2/\alpha)}{2n}},
$$
and thus $\wh{\beta}$ does not belong to (\ref{eq:ci-quant}). When $\wt{\beta}-\beta>\frac{C}{2\sqrt{n(1-\epsilon)^2}}$, we have
\begin{eqnarray}
\nonumber \frac{1}{n}\sum_{i=1}^n\indi\{y_i< \wt{\beta}\} &\geq& \frac{1}{n}\sum_{i=1}^n\indi\left\{y_i< \beta+\frac{C}{2\sqrt{n(1-\epsilon)^2}}\right\} \\
\label{eq:fnf} &\geq& \mathbb{P}_{z\sim (1-\epsilon)N(0,1)+\epsilon Q_0}\left(z< \frac{C}{2\sqrt{n(1-\epsilon)^2}}\right) - \frac{C_1}{\sqrt{n}} \\
\nonumber &\geq& \frac{1}{2} + (1-\epsilon)\mathbb{P}\left(0< N(0,1)< \frac{C}{2\sqrt{n(1-\epsilon)^2}}\right) - \frac{C_1}{\sqrt{n}} \\
\nonumber &\geq& \frac{1}{2}+ \sqrt{\frac{\log(2/\alpha)}{2n}},
\end{eqnarray}
where (\ref{eq:fnf}) holds with high probability by Hoeffding's inequality and the last inequality holds with a sufficiently large $C>0$. The other direction $\frac{1}{n}\sum_{i=1}^n\indi\{y_i\leq \wt{\beta}\} < \frac{1}{2}- \sqrt{\frac{\log(2/\alpha)}{2n}}$ implied by $\wt{\beta}-\beta<-\frac{C}{2\sqrt{n(1-\epsilon)^2}}$ is by a symmetric argument.
\end{proof}

\begin{proof}[Proof of Proposition \ref{prop:cov-gaus}]
The coverage property follows the same argument in the proof of Lemma \ref{lem:self}. For the length guarantee, it suffices to show
$$\frac{\left|\frac{1}{n}\sum_{i=1}^ng(x_i\wt{\beta}-y_i)x_i\right|}{\sqrt{\frac{1}{n}\sum_{i=1}^nx_i^2}}> \sqrt{\frac{2\log(2/\alpha)}{n}}$$
holds for all $|\wt{\beta}-\beta|>\frac{C}{2\sqrt{n(1-\epsilon)^2}}$. This can be established by the same argument in the proof of Theorem \ref{thm:main-length}.
\end{proof}

We note that it is important that Proposition \ref{prop:cov-gaus} states the length guarantee in probability. The following result shows that the in-expectation length is actually infinite in the worst case.
\begin{Proposition}\label{prop:inf-exp-len}
Let the endpoints of the interval (\ref{eq:Z-CI-med-reg-tau}) be $\wh{\beta}_L$ and $\wh{\beta}_R$. Then, for any $C>0$, there is a contamination distribution $Q$ such that $\mathbb{E}|\wh{\beta}_R-\wh{\beta}_L|\geq C$.
\end{Proposition}
\begin{proof}
    We consider $\beta=0$ and thus $y_i=z_i\sim (1-\epsilon)N(0,1)+\epsilon Q$. Since $\mathbb{E}|\wh{\beta}_R-\wh{\beta}_L|\geq \epsilon^n\mathbb{E}_{z_i\overset{iid}{\sim} Q}|\wh{\beta}_R-\wh{\beta}_L|$, we can consider $\epsilon=1$ without loss of generality.
    Let $Q(z_i=A)=Q(z_i=-A)=\frac{1}{2}$ where $A>0$ is a large constant. Since $x_i\sim N(0,1)$, we have $\Prob(\max|x_i|\le 10\sqrt{\log n})\ge 0.99$. Therefore, as long as $|\beta|\le \frac{A}{20\sqrt{\log n}}$, we have $\Prob(\max|x_i\beta|\le \frac{1}{2}|z_i|)\ge 0.99$ and 
    $$ \abs{ \frac{\frac{1}{n}\sum_{i=1}^ng(x_i\beta-y_i)x_i}{\sqrt{\frac{1}{n}\sum_{i=1}^nx_i^2}}}\le  \abs{ \frac{\frac{1}{n}\sum_{i=1}^n \text{sign}(z_i)x_i}{\sqrt{\frac{1} {n}\sum_{i=1}^nx_i^2}}}+ \frac{\frac{1}{n}\sum_{i=1}^n (1-g(A/2))|x_i|}{\sqrt{\frac{1} {n}\sum_{i=1}^nx_i^2}}
    \le \sqrt{\frac{2\log(2/\alpha)}{n}}$$
    with probability at least $0.99(1-2\alpha)$ for large enough $A$ since $ \abs{ \frac{\frac{1}{n}\sum_{i=1}^n \text{sign}(z_i)x_i}{\sqrt{\frac{1} {n}\sum_{i=1}^nx_i^2}}}\le \sqrt{\frac{2\log(1/\alpha)}{n}}$ with probability at least $1-2\alpha$ and $\frac{\frac{1}{n}\sum_{i=1}^n|x_i|}{\sqrt{\frac{1} {n}\sum_{i=1}^nx_i^2}}\le 1$ by Cauchy's inequality. The expected length is thus at least $\frac{0.99A(1-2\alpha)}{10\sqrt{\log n}}$ which can be arbitrarily large as long as $A$ is chosen large enough.
\end{proof}

\subsection{Proofs of Proposition \ref{prop:imp-joint} and Proposition \ref{prop:imp-depend}} \label{sec:pf-by}

\begin{proof}[Proof of Proposition \ref{prop:imp-joint}]
   The existence of $Q$ such that  $P_{0,1,1}=\frac{2}{3}P_{c,\sigma,\sigma}+\frac{1}{3}Q$ is equivalent to 
    \begin{eqnarray*}
      &&  \exp\left(-\frac{x^2+y^2}{2}\right) \ge 
      \frac{2}{3\sigma^2} \exp\left(-\frac{x^2+(y-xc)^2}{2\sigma^2}\right) \text{ for all } x,y\in \mathbb{R}\\
      &\Leftarrow& 2\sigma^2\ln\left( \frac{3\sigma^2}{2}\right)+(1+c^2-\sigma^2)x^2+(1-\sigma^2)y^2-2cxy\ge 0 
      \text{ for all } x,y\in \mathbb{R} \\
      &\Leftarrow& 
      \begin{cases}
          \frac{3\sigma^2}{2}\ge 1 \\
          \begin{pmatrix}
             1+c^2-\sigma^2& -c\\
             -c &1-\sigma^2
          \end{pmatrix}\succeq 0
      \end{cases}
      \\
    &\Leftarrow& 
    \begin{cases}
      \frac{2}{3}\le \sigma^2 <1 \\
      c<\frac{1-\sigma^2}{\sigma}
    \end{cases}.
  \end{eqnarray*}
  It's easy to see $\sigma=0.99,c=\frac{1-0.99^2}{2\times0.99}$ satisfy the conditions.
\end{proof}

We need the following lemma to prove Proposition \ref{prop:imp-depend}.
\begin{Lemma} \label{lm:contaminated normal}
    Suppose $\sigma_1\le \sigma_2\le \frac{\sigma_1}{1-\epsilon}$ and $\abs{r}\le \sqrt{2(\sigma_2^2-\sigma_1^2)\ln\frac{\sigma_1}{(1-\epsilon)\sigma_2}}$. Then there exists distribution $Q$ such that 
    $$ (1-\epsilon)N(0,\sigma_1^2)+\epsilon Q = N(r,\sigma_2^2). $$
\end{Lemma}
\begin{proof}
    By symmetry we can assume $r\ge0$. The existence of $Q$ is equivalent to for all $t\in \mathbb{R}$,
    $$ \frac{1-\epsilon}{\sigma_1} \exp\left(-\frac{t^2}{2\sigma_1^2}\right) \le \frac{1}{\sigma_2} \exp\left(-\frac{(t-r)^2}{2\sigma_2^2}\right), $$
    which is implied by  
    $$ r\le \min_{t\in \mathbb{R}}\left(t+\sqrt{2\sigma_2^2\left(\frac{t^2}{2\sigma_1^2}-\ln\frac{(1-\epsilon)\sigma_2}{\sigma_1}\right)}\right)$$
    when $\sigma_1\le \sigma_2\le \frac{\sigma_1}{1-\epsilon},r\ge0$. By differentiation it's easy to show that 
    $$ \min_{t\in \mathbb{R}}\left(t+\sqrt{2\sigma_2^2\left(\frac{t^2}{2\sigma_1^2}-\ln\frac{(1-\epsilon)\sigma_2}{\sigma_1}\right)}\right)=\sqrt{2(\sigma_2^2-\sigma_1^2)\ln\frac{\sigma_1}{(1-\epsilon)\sigma_2}}. $$   
\end{proof}

\begin{proof}[Proof of Proposition \ref{prop:imp-depend}] We define $Q_0(o_i=0|x_i,z_i)=1$.
   Then, we need to construct the conditional distribution $o_i|x_i,z_i \sim Q_1(x_i,z_i)$ such that $z_i+o_i \sim N(-x_ic,1)$, where $z_i \sim N(0,0.99^2)$ and $c=0.07$. Let  $q_{\alpha}=\Phi^{-1}(1-\alpha)$ be the Gaussian quantile. We describe the generating process of $o_i|x_i,z_i$ as follows.
   \begin{enumerate}
       \item When $|x_i|\le q_{0.05}$, by \prettyref{lm:contaminated normal}, there exists some $Q_{x_i}$ such that 
       $$ \frac{2}{3} N(0,0.99^2)+\frac{1}{3} Q_{x_i} = N(-x_ic,1),$$
        since $|x_ic|\le 0.07q_{0.05} \le \sqrt{2(1-0.99^2)\ln\frac{3\times0.99}{2}}$.  Let $q_i\sim Q_{x_i}$ and $o_i=B_i(q_i-z_i)$ with $B_i \sim \text{Bernoulli}(\frac{1}{3})$. Then $z_i+o_i \sim N(-x_ic,1)$.
       \item When $|x_i|> q_{0.05}$, let $r_i \sim N(-x_ic,1)$ and $o_i=r_i-z_i$. Then $z_i+o_i \sim N(-x_ic,1)$.
   \end{enumerate}
   Overall, the sampling process of $o_i|x_i,z_i \sim Q_1(x_i,z_i)$ can be written as $o_i=\indc{|x_i|\le q_{0.05}}B_i(q_i-z_i)+\indc{|x_i|> q_{0.05}}(r_i-z_i)$. Since 
   $\mathbb{E}Q_1(o_i\neq 0|x_i,z_i)=\frac{1}{3}\Prob(|x_i|\le q_{0.05})+ \Prob(|x_i|> q_{0.05})=0.4$, the construction is valid.
\end{proof}

\subsection{Proofs of Technical Lemmas}\label{sec:pf-lem}

We will prove Lemma \ref{lem:decor}, Lemma \ref{lem:for-gamma}, Lemma \ref{lem:for-main} and Lemma \ref{lem:self} in this section. We first state some concentration bounds that will be needed.

\begin{Lemma}[\cite{hoeffding1963probability}] \label{lm:Hoeffding}
Suppose $x_1,\cdots,x_n$ are independent zero-mean random variables such that $a_i\leq x_i\leq b_i$ for all $i\in[n]$. Then,
$$\mathbb{P}\left(\sum_{i=1}^nx_i\geq t\right)\leq\exp\left(-\frac{2t^2}{\sum_{i=1}^n(b_i-a_i)^2}\right),$$
for any $t>0$.
\end{Lemma}

\begin{Lemma}[Proposition 5.16 of \cite{vershynin2010introduction}] \label{lm:Bernstein}
Suppose $x_1,\cdots,x_n$ are independent zero-mean sub-exponential random variables such that $\max_{i\in[n]}\mathbb{E}\exp(x_i/K)\leq e$ for some constant $K>0$. Then, there exists some constant $C>0$ such that
$$\mathbb{P}\left(\sum_{i=1}^nx_i\geq t\right)\leq \exp\left(-C\left(\frac{t^2}{n}\wedge t\right)\right),$$
for any $t>0$.
\end{Lemma}

\begin{Lemma}[Theorem 2.1 of \cite{al2025sharp}] \label{lm:fourth moment of covariate}
Suppose $X_1,\cdots,X_n$ are i.i.d. vectors satisfying \prettyref{as:bounded moments}. For any constant $\delta\in(0,1)$ and $d\in\{2,4\}$, there exists some constant $C>0$ depending on $\delta$ such that
$$\sup_{v\in S_{p-1}}\left|\frac{1}{n}\sum_{i=1}^n\left((v^TX_i)^d-\mathbb{E}(v^TX_i)^d\right)\right|\leq C\left(\sqrt{\frac{p}{n}}+\frac{p^{d/2}}{n}\right),$$
with probability at least $1-\delta$.
\end{Lemma}

\begin{proof}[Proof of Lemma \ref{lem:decor}]
The expansion (\ref{eq:taylor}) and the arguments in Section \ref{sec:cov} imply
$$|G_n(\beta)|\leq\left|\frac{1}{n}\sum_{i=1}^ng\left(z_i\right)\wt{x}_i\right|+ \|\wh{\gamma}-\gamma\|\left\|\frac{1}{n}\sum_{i=1}^n\left(g'\left(z_i\right)-\mathbb{E}g'\left(z_i\right)\right)\wt{x}_iw_i\right\| + \frac{\|\wh{\gamma}-\gamma\|^2}{2}\opnorm{\frac{1}{n}\sum_{i=1}^n|\wt{x}_i|w_iw_i^T}.$$
This immediately implies the desired bound using Lemma \ref{lem:for-main}.
\end{proof}

\begin{proof}[Proof of Lemma \ref{lem:for-gamma}] 
The data generating process (\ref{eq:noise-sig}) can be equivalently described as follows. We first sample $I_i\overset{iid}\sim \text{Bernoulli}(1-\epsilon)$. Then, sample $z_i|I_i\sim I_iN(0,1)+(1-I_i)Q$. Define the set $I=\{i\in [n]:I_i=1,|z_i|\le 0.5\}$ and $n_0=|I|$. Besides, we define $\wc{x}_i=x_i-\alpha^Tw_i$ and $\wc{X}_i^T=(\wc{x}_i,w_i^T)^T$. 
Since the $(1,1)$ entry of $\Sigma^{-1}$ equals $(\Sigma_{xx}-\Sigma_{xw}\Sigma_{ww}^{-1}\Sigma_{wx})^{-1}$, $\Sigma_{xw}\Sigma_{ww}^{-1}\Sigma_{wx}\le \Sigma_{xx}\le \opnorm{\Sigma}\le 2K_1^2 $ and $$\norm{\alpha}^2=\Sigma_{xw}\Sigma_{ww}^{-2}\Sigma_{wx}\le \norm{\Sigma_{ww}^{-\frac{1}{2}}\Sigma_{wx}}^2 \opnorm{\Sigma_{ww}^{-1}}=\Sigma_{xw}\Sigma_{ww}^{-1}\Sigma_{wx}\opnorm{\Sigma_{ww}^{-1}}\le 2K_1^2K_3.$$
Thus, we can write $\widecheck{x}_i$ as $u^TX_i$ with $\|u\|$ bounded by some constant. This implies the vector $\widecheck{X}_i$ also satisfies \prettyref{as:bounded moments} with constants $\widecheck{K}_1$, $\widecheck{K}_2$ and $\widecheck{K}_3$.

By triangle inequality, we have 
$$  \norm{\frac{1}{n}\sum_{i=1}^n \rho'(z_i) \wt{X}_i}\le 
   \norm{\frac{1}{n}\sum_{i=1}^n \rho'(z_i) \wc{X}_i}+\norm{\frac{1}{n}\sum_{i=1}^n \rho'(z_i) (\wt{X}_i-\wc{X}_i)}.$$
Since $\widecheck{X}_i$ is symmetric and is independent of  $z_i$, and thus
\begin{equation}
 \Expect \norm{\frac{1}{n}\sum_{i=1}^n \rho'(z_i) \wc{X}_i}^2 
   = \sum_{j=1}^p \Expect \left(\frac{1}{n}\sum_{i=1}^n \rho'(z_i) \wc{X}_{ij}\right)^2 
   = \frac{1}{n^2} \sum_{j=1}^p  \sum_{i=1}^n \Expect
   \bpar{\rho'(z_i) \wc{X}_{ij}}^2 \le C_1\frac{p}{n},\label{eq:expell2}
   \end{equation}
which implies $\norm{\frac{1}{n}\sum_{i=1}^n \rho'(z_i) \wc{X}_i}\le C_2\sqrt{\frac{p}{n}}$ with high probability by Markov's inequality. By \prettyref{lm:error in first-stage}, we have
$$ \norm{\frac{1}{n}\sum_{i=1}^n \rho'(z_i) (\wt{X}_i-\wc{X}_i)}=\abs{\frac{1}{n}\sum_{i=1}^n \rho'(z_i) (\wh{\alpha}-\alpha)^Tw_i} \le \norm{\wh{\alpha}-\alpha}\norm{\frac{1}{n}\sum_{i=1}^n \rho'(z_i)w_i}\le C_3\sqrt{\frac{p}{n}},$$
with high probability.
Combining the above displays, we obtain the first inequality in \prettyref{lem:for-gamma}.

   For the second inequality,  we note that $n(1-\epsilon)^2$ is large enough when  $\frac{p^2}{n(1-\epsilon)^4}$ is small. By Chebyshev's inequality, we have 
    $$ \left|\frac{n_0}{n}-(1-\epsilon)(\Phi(0.5)-\Phi(-0.5))\right| 
    \le C_4\sqrt{\frac{\Var(\indc{|z_i|\le 0.5, I_i=1})}{n}}
    \le \frac{(1-\epsilon)(\Phi(0.5)-\Phi(-0.5))}{2},$$
    with high probability, which implies
    \begin{equation}
    \frac{n_0}{n}\ge \frac{1}{2}(1-\epsilon)(\Phi(0.5)-\Phi(-0.5)).\label{eq:nling}
    \end{equation}
     For any $v\in S_{p-1}$ and some constant $\kappa$ to be specified later, by Lipschitz continuity we have 
    \begin{eqnarray}
        \nonumber &&\frac{1}{n_0} \sum_{i\in I} |\wt{X}_i^T v|\land|\wt{X}_i^T v|^2  \\
      \nonumber   &\ge& \frac{1}{n_0} \sum_{i\in I}  |\wc{X}_i^T v|\land|\wt{X}_i^T v|^2 -\frac{1}{n_0}  \sum_{i\in I}  \abs{v^T(\wc{X}_i-\wt{X}_i)} \\
     \label{eq:liproot}   &\ge& \frac{1}{n_0} \sum_{i\in I}  |\wc{X}_i^T v|\land|\wc{X}_i^T v|^2 -\frac{1}{n_0}  \sum_{i\in I}  \abs{v^T(\wc{X}_i-\wt{X}_i)}-\frac{1}{n_0}\sum_{i\in I}  \abs{v^T(\wc{X}_i-\wt{X}_i)}2\sqrt{\abs{v^T\wc{X}_i}} \\
     \nonumber    &\ge& \frac{1}{n_0} \sum_{i\in I}  |\wc{X}_i^T v|\land|\wc{X}_i^T v|^2 -\frac{1}{n_0}  \sum_{i\in I}  \abs{v^T(\wc{X}_i-\wt{X}_i)}-\frac{1}{n_0\kappa}\sum_{i\in I}  \abs{v^T(\wc{X}_i-\wt{X}_i)}^2-\frac{\kappa}{n_0}  \sum_{i\in I}\abs{v^T\wc{X}_i} \\
     \nonumber    &\ge& \frac{1}{n_0} \sum_{i\in I}  |\wc{X}_i^T v|\land|\wc{X}_i^T v|^2 -\sqrt{\frac{1}{n_0}  \sum_{i\in I}  \abs{v^T(\wc{X}_i-\wt{X}_i)}^2}-\frac{1}{n_0\kappa}\sum_{i\in I}  \abs{v^T(\wc{X}_i-\wt{X}_i)}^2-\kappa \sqrt{\frac{1}{n_0}  \sum_{i\in I}\abs{v^T\wc{X}_i}^2}.
    \end{eqnarray}
    The inequality (\ref{eq:liproot}) uses the fact that $f(x)=x^2\wedge a$ is $\sqrt{2a}$-Lipschitz.
   Let $\mathcal{V}$ be a covering of $S_{p-1}$ such that for any $v\in S_{p-1}$, there exists some $v'\in \mathcal{V}$ s.t. 
    $\norm{v-v'}\le \Delta<1$. It is known that $\mathcal{V}$ can be constructed such that  $|\mathcal{V}|\le \bpar{\frac{3}{\Delta}}^p$. By the same argument used in the above display, we have 
    \begin{eqnarray*}
       &&\frac{1}{n_0} \sum_{i\in I}  |\wc{X}_i^T v|\land|\wc{X}_i^T v|^2 \\
       &\ge&  \frac{1}{n_0} \sum_{i\in I}  |\wc{X}_i^T v'|\land|\wc{X}_i^T v'|^2 -\sqrt{\frac{1}{n_0}  \sum_{i\in I}  \abs{(v-v')^T\wc{X}_i}^2}-\frac{1}{n_0\kappa}\sum_{i\in I}  \abs{(v-v')^T\wc{X}_i}^2-\kappa \sqrt{\frac{1}{n_0}  \sum_{i\in I}\abs{v'^T\wc{X}_i}^2}
    \end{eqnarray*}
    Combining the above two displays, we have 
    \begin{eqnarray*}
      &&\inf_{\norm{v}=1}\frac{1}{n_0} \sum_{i\in I} |\wt{X}_i^T v|\land|\wt{X}_i^T v|^2 \ \\
      &\ge&
      \inf_{\norm{v}\in \mathcal{V}}\frac{1}{n_0} \sum_{i\in I}  |\wc{X}_i^T v'|\land|\wc{X}_i^T v'|^2-\sqrt{\frac{1}{n_0}  \sum_{i\in I}  \abs{(\hat{\alpha}-\alpha)^Tw_i}^2}-\frac{1}{n_0\kappa}\sum_{i\in I}  \abs{(\hat{\alpha}-\alpha)^Tw_i}^2-\kappa \sqrt{\opnorm{\wc{\Sigma}}}   \\
      && -\Delta \sqrt{\opnorm{\wc{\Sigma}}}-\frac{\Delta^2}{\kappa}\opnorm{\wc{\Sigma}}-\kappa \sqrt{\opnorm{\wc{\Sigma}}},
    \end{eqnarray*}
    where $\widecheck{\Sigma}=\frac{1}{n_0}\sum_{i\in I}\widecheck{X}_i\widecheck{X}_i^T$. By (\ref{eq:nling}), we know that $\frac{p}{n_0}\le \frac{p}{n} \frac{2}{(1-\epsilon)(\Phi(0.5)-\Phi(-0.5))}$ is small when $\frac{p^2}{n(1-\epsilon)^4}$ is small. Since $n_0$ and $I$ are independent of $\{\widecheck{X}_i\}_{i=1}^n$, we have $\opnorm{\wc{\Sigma}}\le \opnorm{\wc{\Sigma}-\Cov(\wc{X}_i)}+ \opnorm{\Cov(\wc{X}_i)}\le C_5 $ by Lemma \ref{lm:fourth moment of covariate} and the sub-Gaussianity of $\widecheck{X}_i$. By \prettyref{lm:error in first-stage} and Markov's inequality, we have 
    $\frac{1}{n_0}\sum_{i\in I}  \abs{(\hat{\alpha}-\alpha)^Tw_i}^2\le \norm{\hat{\alpha}-\alpha}^2\frac{1}{n_0}\sum_{i\in I} \norm{w_i}^2 \le C_6\frac{p^2}{n}$. In addition,
    $$
       \Expect (|\wc{X}_i^Tv|\land|\wc{X}_i^Tv|^2) 
       \ge  \left(\frac{1}{\wc{K}_2}\land 1 \right) \Expect (|\wc{X}_i^Tv|^2 \indi\{|v^T\wc{X}_i|\leq \wc{K}_2\}) 
       \ge \frac{\wc{K}_3}{\wc{K}_2}\land \wc{K}_3.
    $$
    Since $|\wc{X}_i^Tv|\land|\wc{X}_i^Tv|^2$ is sub-exponential, by \prettyref{lm:Bernstein} and union bound there exists some constant $C_7$ such that
    $$ \sup_{\norm{v}\in \mathcal{V}}\abs{\frac{1}{n_0}  \sum_{i\in I}  |\wc{X}_i^T v'|\land|\wc{X}_i^T v'|^2-\Expect (|\wc{X}_i^Tv|\land|\wc{X}_i^Tv|^2) }\le C_7\left(\sqrt{\frac{p\log(1/\Delta)}{n_0}}+\frac{p\log(1/\Delta)}{n_0}\right)$$
    with high probability.
     In the end, choosing $\kappa=\Delta$ to be a sufficiently small constant, then when $\frac{p^2}{n(1-\epsilon)^4}$ is small, there exists some constant $c>0$ such that 
    $$ \inf_{\norm{v}=1}\frac{1}{n_0} \sum_{i\in I} |\wt{X}_i^T v|\land|\wt{X}_i^T v|^2 \ge c $$
    with high probability by combining the above bounds with (\ref{eq:nling}). Using (\ref{eq:nling}) again, we obtain 
$$\inf_{\norm{v}=1}\frac{1}{n} \sum_{i=1}^n \left(|\wt{X}_i^T v|\land|\wt{X}_i^T v|^2 \right) \indc{|z_i|\le 1}\geq \frac{n_0}{n} \inf_{\norm{v}=1}\frac{1}{n_0} \sum_{i\in I} |\wt{X}_i^T v|\land|\wt{X}_i^T v|^2\geq c_1(1-\epsilon).$$
The proof is complete.
\end{proof}

\begin{proof}[Proof of Lemma \ref{lem:for-main}]
Recall the definitions of $\widecheck{x}_i$ and $\widecheck{X}_i$ in the proof of Lemma \ref{lem:for-gamma}.
   By triangle inequality,
   $$ \left\|\frac{1}{n}\sum_{i=1}^n\left(g'\left(z_i\right)-\mathbb{E}g'\left(z_i\right)\right)\wt{x}_iw_i\right\|\le \left\|\frac{1}{n}\sum_{i=1}^n\left(g'\left(z_i\right)-\mathbb{E}g'\left(z_i\right)\right)\wc{x}_iw_i\right\|+\left\|\frac{1}{n}\sum_{i=1}^n\left(g'\left(z_i\right)-\mathbb{E}g'\left(z_i\right)\right)(\wt{x}_i-\wc{x}_i)w_i\right\|. $$
   Using a similar argument to (\ref{eq:expell2}) together with Markov's inequality, we have
   $$\left\|\frac{1}{n}\sum_{i=1}^n\left(g'\left(z_i\right)-\mathbb{E}g'\left(z_i\right)\right)\wc{x}_iw_i\right\|\le C_1\sqrt{\frac{p}{n}}.$$
   Moreover, Cauchy-Schwarz and \prettyref{lm:error in first-stage} imply
   \begin{eqnarray*}
      &&\left\|\frac{1}{n}\sum_{i=1}^n\left(g'\left(z_i\right)-\mathbb{E}g'\left(z_i\right)\right)(\wt{x}_i-\wc{x}_i)w_i\right\| \\
      &\le& \norm{\hat{\alpha}-\alpha} \opnorm{\frac{1}{n}\sum_{i=1}^n\left(g'\left(z_i\right)-\mathbb{E}g'\left(z_i\right)\right)w_iw_i^T}  \\
      &\le& C_2 \sqrt{\frac{p}{n}} \opnorm{\frac{1}{n}\sum_{i=1}^n\left(g'\left(z_i\right)-\mathbb{E}g'\left(z_i\right)\right)w_iw_i^T},
   \end{eqnarray*}
   where
   \begin{eqnarray*}
     &&\opnorm{\frac{1}{n}\sum_{i=1}^n\left(g'\left(z_i\right)-\mathbb{E}g'\left(z_i\right)\right)w_iw_i^T}\\
     &=& \sup_{\norm{v}=1} \frac{1}{n}\sum_{i=1}^n\left(g'\left(z_i\right)-\mathbb{E}g'\left(z_i\right)\right)(v^Tw_i)^2 \\
     &\le& \frac{1}{2n}\sum_{i=1}^n\left(g'\left(z_i\right)-\mathbb{E}g'\left(z_i\right)\right)^2+\sup_{\norm{v}=1}\frac{1}{2n}\sum_{i=1}^n(v^Tw_i)^4 \\
     &\le&  C_3,
   \end{eqnarray*}
   by \prettyref{lm:fourth moment of covariate}.
   Combining the above high-probability bounds, we obtain the first inequality of \prettyref{lem:for-main}.
 
   The second inequality follows directly from \prettyref{lm:fourth moment of covariate} and the fourth inequality in \prettyref{lem:for-main} which will be proved later,  
   $$ \opnorm{\frac{1}{n}\sum_{i=1}^n|\wt{x}_i|w_iw_i^T} =\sup_{\norm{v}=1} \frac{1}{n}\sum_{i=1}^n (v^Tw_i)^2 |\wt{x}_i| 
   \le \frac{1}{2}\left(\frac{1}{n}\sum_{i=1}^n \wt{x}_i^2+ \sup_{\norm{v}=1} \frac{1}{n}\sum_{i=1}^n (v^Tw_i)^4 \right)
   \le C_4.$$
   
To prove the third inequality of \prettyref{lem:for-main}, we first claim that
\begin{equation}
\max_{1\le i\le n}\norm{w_i} \leq C_5(\sqrt{p}+\sqrt{\log n}) \label{eq:maximumw}
\end{equation}
holds with high probability. To show (\ref{eq:maximumw}), recall the $\Delta$-net $\mathcal{V}$ used in the proof of Lemma \ref{lem:for-gamma}. With $\Delta$ being a sufficiently small constant, we have $\max_{1\le i\le n}\norm{w_i}\leq 2\max_{1\le i\le n}\max_{v\in\mathcal{V}}|v^Tw_i|$, and (\ref{eq:maximumw}) follows by a union bound argument with \prettyref{as:bounded moments}. Combining (\ref{eq:maximumw}) with \prettyref{lem:estimation-huber}, we have
\begin{equation}
\max_{1\leq i\leq n}|(\gamma-\wh{\gamma})^Tw_i|\le \norm{\gamma-\wh{\gamma}}\max_{1\leq i\leq n}\norm{w_i}\le C_6\sqrt{\frac{p}{n(1-\epsilon)^2}}(\sqrt{p}+\sqrt{\log n})\le \frac{1}{2}. \label{eq:maxiinnerp}
\end{equation}
The event (\ref{eq:maxiinnerp}) implies
    $$ \frac{1}{n} \sum_{i=1}^n(|\wt{x}_i|\wedge\wt{x}_i^2)\indc{|(\gamma-\wh{\gamma})^Tw_i+z_i|\leq 1} 
    \nonumber\ge \frac{1}{n} \sum_{i=1}^n (|\wt{x}_i|\wedge\wt{x}_i^2)\indc{|z_i|\leq 0.5} 
     \nonumber,$$
which is lower bounded by $c(1-\epsilon)$ following the same analysis leading to the second inequality of Lemma \ref{lem:for-gamma}.

   For the last inequality of \prettyref{lem:for-main}, we first use triangle inequality and Cauchy-Schwarz, and
   $$\abs{\sqrt{\frac{1}{n}\sum_{i=1}^n\wt{x}_i^2}-\sqrt{\frac{1}{n}\sum_{i=1}^n\wc{x}_i^2}}\le \sqrt{\frac{1}{n}\sum_{i=1}^n(\wt{x}_i-\wc{x}_i)^2}\le \norm{\alpha-\hat{\alpha}}\sqrt{\frac{1}{n}\sum_{i=1}^n\norm{w_i}^2 }, $$
   which is at most $C_7\sqrt{\frac{p^2}{n}}$ by \prettyref{lm:error in first-stage}. Under \prettyref{as:bounded moments}, it is straightforward to check that $\frac{1}{n}\sum_{i=1}^n\wc{x}_i^2\asymp 1$. The proof is complete.
\end{proof}

\begin{proof}[Proof of Lemma \ref{lem:self}]
Let $\xi_i\overset{i.i.d}{\sim}\text{Rademacher}$ and define $X_i^s=X_i\xi_i$. Since $X_i$ is symmetric, for any measurable set $A$,
$$\Prob(X_i^s \in A|\xi_i=1)=\Prob(X_i \in A)=\Prob(-X_i \in A)=\Prob(X_i^s \in A|\xi_i=-1),$$ which implies $X_i^s$'s are independent of $\xi_i$'s. Using the notation $\wt{x}_i^s=\xi_i\wt{x}_i$, we have
$$ \frac{\frac{1}{n}\sum_{i=1}^ng(z_i)\wt{x}_i}{\sqrt{\frac{1}{n}\sum_{i=1}^n \wt{x}_i^2}}=\frac{\frac{1}{n}\sum_{i=1}^n\xi_ig(z_i)\wt{x}_i^s}{\sqrt{\frac{1}{n}\sum_{i=1}^n(\wt{x}_i^s)^2}}.$$  
Define $u_i=\frac{g(z_i)\wt{x}_i^s}{\sqrt{\frac{1}{n}\sum_{i=1}^n(\wt{x}_i^s)^2}}$, and then by \prettyref{lm:Hoeffding}, we have 
$$\mathbb{P}\left(\abs{\sum_{i=1}^n \xi_iu_i}\geq t\Bigg|u\right)\leq 2\exp\left(-\frac{t^2}{2\sum_{i=1}^n u_i^2}\right)\le 2\exp\left(-\frac{t^2}{2n}\right).$$
for any $t>0$. The same bound holds for $\mathbb{P}\left(\abs{\sum_{i=1}^n \xi_iu_i}\geq t\right)$ as well.
By verifying that $2\exp\left(-\frac{t^2}{2n}\right)\le 0.97\alpha$ with $t=1.45\sqrt{n\log\frac{2}{\alpha}}$, the proof is complete.
\end{proof}

\vskip 0.2in
\bibliographystyle{apalike}
\bibliography{reference}

\end{sloppypar}

\end{document}